\documentclass[12pt]{article} 
\usepackage{amsmath,amssymb,amsthm}
\usepackage{amsfonts}

\usepackage{mathtools}
\usepackage{color}
\usepackage[nohead,margin=1.0in]{geometry}
\usepackage{booktabs}
\usepackage{graphicx,caption,subcaption}
\usepackage{url}
\usepackage{grffile}
\usepackage{float}
\usepackage{setspace}
\usepackage{hyperref}
\usepackage{textgreek} 
\usepackage[ruled]{algorithm2e}
\usepackage[noend]{algpseudocode}
\theoremstyle{plain}
\newtheorem{theorem}{Theorem}[section]
\newtheorem{definition}[theorem]{Definition}
\newtheorem{lemma}[theorem]{Lemma}
\newtheorem{assumption}[theorem]{Assumption}

\theoremstyle{definition}
\newtheorem{remark}[theorem]{Remark}
\numberwithin{equation}{section}

\newcommand{\R}{\mathbb{R}}

\def\mydot{\hspace{-2pt} \cdot \hspace{-2pt}}

\newcommand{\commentout}[1]{}

\title{Regularization Theory of the Analytic~Deep~Prior~Approach}
\author{Clemens Arndt\thanks{Center for Industrial Mathematics, University of Bremen, 
Germany (\texttt{carndt@uni-bremen.de})}
}
\date{March 2022}

\begin{document}
\maketitle

\begin{abstract}

The analytic deep prior (ADP) approach was recently introduced for the theoretical analysis of deep image prior (DIP) methods with special network architectures. In this paper, we prove that ADP is in fact equivalent to classical variational Ivanov methods for solving ill-posed inverse problems. Besides, we propose a new variant which incorporates the strategy of early stopping into the ADP model. For both variants, we show how classical regularization properties (existence, stability, convergence) can be obtained under common assumptions.
\end{abstract}

%
%

\section{Introduction}

In particular the field of image processing (e.g.\ denoising, deblurring) is a constant source for challenging inverse problems.
The restoration of a corrupted image is typically ill-posed, so regularization techniques are needed to obtain a natural looking result. In other words, the restoration method should incorporate some prior knowledge about the appearance of natural images. 
However, dependent on the application it can be very difficult to give a mathematically exact definition of what natural looking images are.
This makes it hard to encode such prior knowledge in a penalty term  for classical variational regularization approaches (e.g.\ TV regularization \cite{Rudin1992tv, chambolle2011}). 

However, deep learning methods with convolutional neural networks have proven to be quite successful in generating and restoring images \cite{jain2009, goodfellow2014gan, Tai2017}. One reason for that is the use of appropriate training data, but \cite{ulyanov2018dip} shows that just the architecture of an untrained network can already serve as an image prior. The so-called deep image prior (DIP) approach consists in optimizing the weights of a neural network $\varphi_\theta$ to minimize the loss function
\begin{equation} \label{eq_dip_loss}
    \frac{1}{2} \|A \varphi_\theta(z) - y^\delta\|^2
\end{equation}
for some forward operator $A$ and noisy data $y^\delta$ (the network's input $z$ is randomly chosen and kept fixed). Although no training data and no penalty functional is used, DIP produces remarkable results in different image processing tasks, as can be seen in \cite{ulyanov2018dip}. Even challenging problems like sparse angle computed tomography \cite{Baguer_2020} or compressive sensing \cite{heckel2020compressive} can be solved this way.

Developing regularization theory for deep learning methods is of high interest \cite{arridge_maass_2019}. A very prosporous approach is to combine classical theory with deep learning (e.g.\ \cite{Li_2020, romano2017}).
The number of papers which analyze DIP from a theoretical point of view is also growing.
In \cite{habring2021generative} a functional is constructed, which measures the ability of the neural network $\varphi_\theta$ to approximate an arbitrary image. This functional can then be used as a penalty term in a classical variational method. 
The authors of \cite{Shi_Mettes_2022} analyze how fast a DIP network approximates the low-frequency and high-frequency components of the target image. By controlling this so-called spectral bias, overfitting is avoided.
In \cite{heckel2020convolution} the ability to denoise images is attributed to convolutional layers, which are faster in fitting smooth images than noisy ones. The role and the choice of hyperparameters for DIP approaches is described in \cite{Sun_Zhao_2021}. A Bayesian perspective is presented in \cite{Cheng_2019_CVPR}, where DIP is interpreted as a Gaussian process.

The choice of architecture is crucial for applications of DIP. Generative neural networks are a natural choice due to their ability to reproduce natural looking images. But the authors of \cite{otero2020} took a LISTA-like network \cite{gregor2010} instead to develop the so-called analytic deep prior (ADP) approach. This may not lead to a better practical performance of DIP, but it's the foundation for an interesting theory. The main aspect consists in interpreting the training of a neural network as the optimization of a Tikhonov functional. There is an analogy to \cite{alberti2021learning}, where the penalty term for a Tikhonov functional is optimized. But in contrast to that, the focus of \cite{otero2020} is on the forward operator inside the functional (see section \ref{ch_prelim}).

This work summarizes deeper investigations of the ADP model. The main result (Theorem \ref{theo_equivalence}) is an equivalence between the ADP approach and classical Ivanov methods \cite{Vasin1970RelationshipOS}. Out of this follows a complete analysis of the regularization properties of ADP including the existence of solutions, stability of reconstructions and convergence towards the ground truth for vanishing noise. 

In practical applications of DIP, gradient descent and early stopping is used to minimize the loss function \eqref{eq_dip_loss}. Thus, a global (or at least a local) minimum is not reached in general. While this fact was not considered in the theoretical derivation of ADP, we propose a new variant (called ADP-$\beta$) which incorporates the effect of early stopping into the model (section \ref{ch_earlystopping}). We also analyze the regularization properties of this new approach.

In section \ref{ch_numeric} we compare different numerical ways to compute ADP and DIP (with a LISTA-like architecture) solutions of simple inverse problems.\footnote{Code available at \url{https://gitlab.informatik.uni-bremen.de/carndt/analytic_deep_prior}} We find that numerical solutions of both methods are mostly similar to each other, which is important for using the ADP theory for interpretations of DIP. But there can also be observed some interesting disparities between the different numerical ways.
This illustrates a crucial difference between the analytical definition of DIP as a minimization problem and the numerical implementation as a gradient descent iteration.

%
%

\section{Preliminaries and methods} \label{ch_prelim}

We consider an inverse problem based on the operator equation
\begin{equation} \label{eq_inverseproblem}
    A x^\dagger = y^\dagger
\end{equation}
where we want to recover the unknown ground truth $x^\dagger$ as good as possible. The data $y^\dagger$ is typically not known exactly, but we have only access to noisy data $y^\delta$. 

\begin{assumption} \label{def_assumptions}
We make the following assumptions for the inverse problem \eqref{eq_inverseproblem}.
\begin{itemize}
    \item Let $X, Y$ be Hilbert spaces and $A \in L(X,Y)$.
    \item There exists $x^\dagger \in X$ and for a given $\delta >0$, it holds $\|y^\delta - y^\dagger\| \leqslant \delta$ for $y^\delta \in Y$.
    \item Let $R \colon X \to [0, \infty]$ be a convex, coercive and weakly lower semicontinuous functional with $R \not\equiv \infty$.
\end{itemize}
\end{assumption}

We recall the definition of Bregman distances, which we will use in Theorem \ref{theo_convergencebeta} for a convergence result, similar to the ones in \cite{Burger_2004, Hofmann_2007}.

\begin{definition}[Bregman distance]
For a convex functional $R\colon X \to [0, \infty]$ with subdifferential $\partial R$ and $\tilde{x}, x \in X$, the Bregman distance is defined as the set
\begin{equation}
D_R(\tilde{x}, x) = \{R(\tilde{x}) - R(x) - \langle p, \tilde{x} - x \rangle \, | \, p \in \partial R(x)\}.
\end{equation} 
\end{definition}

The DIP approach (introduced by \cite{ulyanov2018dip}) for the inverse problem \eqref{eq_inverseproblem} consists in solving
\begin{equation} \label{eq_DIP}
    \min_{\theta} \frac{1}{2} \|A \varphi_\theta (z) - y^\delta\|^2
\end{equation}
via a gradient descent w.r.t.\ the parameters $\theta$ of a neural network $\varphi_\theta$, as already described in the introduction. Despite the use of a neural network, DIP is a model-based approach and not data-based. To derive the ADP approach, we have to make two assumptions (see \cite{otero2020} for details).

The first one is choosing $\varphi_\theta$ to be a LISTA-like network \cite{gregor2010}, which consists of several layers of the form
\begin{equation} \label{eq_lista}
    x^{l+1} = S_{\alpha \lambda}(x^l - \lambda B^* (B x^l - y^\delta)),
\end{equation}
where $B=\theta$ is the trainable parameter. Originally, this architecture is inspired by ISTA \cite{daubechies2004}, an algorithm for finding sparse solutions of the inverse problem \eqref{eq_inverseproblem}, and $S_{\alpha \lambda}$ is the shrinkage function. More general, we can choose $S_{\alpha \lambda}$ to be the proximal mapping of a penalty functional $R$. Then, \eqref{eq_lista} equals a proximal forward-backward splitting algorithm \cite[Theorem 3.4]{combettes2005} which converges to the solution of the minimization problem  
\begin{equation} \label{eq_minproblemB}
    \min_{x \in X} \frac{1}{2} \|B x - y^\delta\|^2 + \alpha R(x).
\end{equation}

The second assumption is letting the number of layers tend to infinity. This might be difficult in practice (see section \ref{ch_numeric}), but it causes the output $\varphi_\theta(z)$ of the network to be a solution of \eqref{eq_minproblemB}. Therefore the ADP model (introduced by \cite{otero2020}) is defined as
\begin{align} \label{eq_ADP}
\begin{split}
    &\min_{B \in L(X,Y)} \frac{1}{2} \|A x(B) - y^\delta\|^2 \\
    \text{s.t.} \quad x(B) = \, &\mathrm{arg} \min_{x \in X} \frac{1}{2} \|B x - y^\delta\|^2 + \alpha R(x).
\end{split}
\end{align}
While DIP is about optimizing the weights of a neural network, ADP is about optimizing the forward operator in a Tikhonov functional.
If we add an additional regularization term for the operator $B$, we get the (new) ADP-$\beta$ model
\begin{equation}
\begin{split} \label{eq_ADPbeta}
    \min_{B \in L(X,Y)} \frac{1}{2} \| A x(B) - y^\delta\|^2  + \beta \|B-A\|^2 \\
    \text{s.t.} \quad x(B) =  \mathrm{arg} \min_{x \in X} \, \frac{1}{2} \|Bx - y^\delta\|^2 + \alpha R(x).
\end{split}
\end{equation}
The reason for this modification will be explained in section \ref{ch_earlystopping}.

To guarantee uniqueness of $x(B)$, the functional $R$ should be strictly convex, but this is not always required. 
If we assume $R$ even to be strongly convex, $x(B)$ depends continuously on $B$ as the following theorem states. It will be useful for proving existence and stability results for ADP-$\beta$.

\begin{theorem} \label{theo_xBstetig}
Let $R \colon X \to [0, \infty]$ be a strongly convex, coercive and weakly lower semicontinuous functional. Then
\begin{equation}
    x(B) = \mathrm{arg} \min_{x \in X} \frac{1}{2} \|B x - y^\delta\|^2 + \alpha R(x)
\end{equation} 
depends continuously on $B \in L(X,Y)$.
\end{theorem}

The proof can be found in the appendix \ref{app_proof_xBstetig}.

%
%

\section{Theoretical results} \label{ch_theoreticalresults}

%
%

\subsection{Equivalence to classical methods} \label{ch_equivalence}
DIP solutions of inverse problems are naturally restricted to be the output of a neural network. Analogously, only elements of the set
\begin{equation}
    U_{\alpha R} = \left\lbrace \hat{x} \in X \, \bigg| \, \exists B \in L(X,Y) : \hat x = \mathrm{arg} \min_{x \in X} \, \frac{1}{2} \|Bx - y^\delta\|^2 + \alpha R(x) \right\rbrace
\end{equation}
can be solutions of the ADP approach. By definition
\begin{align}
&\min_{x \in U_{\alpha R}} \frac{1}{2} \| A x - y^\delta\|^2 
\end{align}
is equivalent to the original ADP problem \eqref{eq_ADP}. To get a better understanding of this minimization problem we investigate $U_{\alpha R}$. It will turn out that the set $U_{\alpha R}$ can be characterized in a much easier way, even without using an operator $B \in L(X,Y)$. For this purpose, we formulate the following lemmas. 

\begin{lemma} \label{lem_existsB_allg}
Let $R \colon X \to [0,\infty]$ be a convex, coercive and weakly lower semicontinuous functional and $\hat{x} \in X$, $y^\delta \in Y$, $y^\delta \neq 0$ and $\alpha>0$ be arbitrary. If there exists $v \in \partial R(\hat{x})$ such that
\begin{equation}
 \alpha \langle v,\hat{x}\rangle \leqslant \frac{\|y^\delta\|^2}{4}
\end{equation}
holds, then there exists a linear operator $B \in L(X,Y)$ which fulfills
\begin{equation}
\hat{x} = \mathrm{arg} \min_{x \in X} \frac{1}{2} \|B x - y^\delta\|^2 + \alpha R(x).
\end{equation}
\end{lemma}

The proof can be found in the appendix \ref{app_proof_existsB_allg}.

\begin{lemma} \label{lem_notexistsB_allg}
Let $R \colon X \to [0,\infty]$ be a convex, coercive and weakly lower semicontinuous functional and $\hat{x} \in X$, $y^\delta \in Y$, $\alpha>0$ be arbitrary. If for every $v \in \partial R(\hat{x})$
\begin{equation}
\alpha \langle v,\hat{x}\rangle > \frac{\|y^\delta\|^2}{4}
\end{equation}
holds, then there exists no linear operator $B \in L(X,Y)$ which fulfills
\begin{equation}
\hat{x} = \mathrm{arg} \min_{x \in X} \frac{1}{2} \|B x - y^\delta\|^2 + \alpha R(x).
\end{equation}
\end{lemma}

The proof can be found in the appendix \ref{app_proof_notexistsB_allg}. For given $y^\delta \in Y$, $\hat{x} \in X$ and a penalty term $R$, these lemmas state whether there exists a linear forward operator $B \colon X \to Y$ such that $\hat{x}$ is the Tikhonov solution w.r.t. $R$ of the inverse problem w.r.t.\ $y^\delta$. As a consequence, we can write the ADP minimization problem with a much simpler side constraint.

\begin{theorem} \label{theo_equivalence}
Let Assumption \ref{def_assumptions} hold.
Then, for all $y^\delta \in Y$, $y^\delta \neq 0$, $\alpha>0$ the formulation
\begin{align} \label{eq_altForm_allg}
\begin{split}
&\min_{x \in X} \frac{1}{2} \| A x - y^\delta\|^2 \\
\text{s.t.} \quad &\exists v \in \partial R(x): \quad \alpha\langle v,x\rangle \leqslant \frac{\|y^\delta\|^2}{4}
\end{split}
\end{align}
is equivalent to the ADP-Problem \eqref{eq_ADP}.
\end{theorem}

\begin{proof}
According to Lemma \ref{lem_existsB_allg} and Lemma \ref{lem_notexistsB_allg} there exists a linear operator $B \in L(X,Y)$ such that
\begin{equation}
    \hat{x} = \mathrm{arg} \min_{x \in X} \frac{1}{2} \|B x - y^\delta\|^2 + \alpha R(x)
\end{equation}
if and only if $\hat{x}$ fulfills the side constraint of \eqref{eq_altForm_allg}.
\end{proof}

\begin{remark} \label{rem_l2_equivalence}
For the standard Tikhonov penalty term $R(x) = \frac{1}{2}\|x\|^2$, it holds $\partial R(x) = x$. In this case we get
\begin{align}
\begin{split} \label{eq_altFormtik}
\min_{x \in X} \frac{1}{2}\|A x - y^\delta\|^2\\
\text{s.t.} \quad \|x\|^2 \leqslant \frac{\|y^\delta\|^2}{4 \alpha}
\end{split}
\end{align}
as an equivalent formulation of the ADP problem \eqref{eq_ADP}. For $r = \|y^\delta\|^2/(4\alpha)$, this equals the Ivanov regularization method 
\begin{equation}
\begin{split} \label{eq_Ivanov}
\min_{x \in X} \frac{1}{2}\|A x - y^\delta\|^2\\
\text{s.t.} \quad \|x\|^2 \leqslant r.
\end{split}
\end{equation}
As \cite{Vasin1970RelationshipOS} shows, this method is in fact equivalent to the Tikhonov method
\begin{equation}
    \min_{x \in X} \frac{1}{2}\|A x - y^\delta\|^2 + \frac{\tilde{\alpha}}{2} \|x\|^2
\end{equation}
for some $\tilde{\alpha}$ dependent on $y^\delta$ and $r$. We note that the Tikhonov parameter $\tilde{\alpha}$ may be equal to zero and in particular it differs from the parameter $\alpha$ of the ADP problem (see section \ref{ch_earlystopping}). 
\end{remark}

\begin{remark} \label{rem_bad_penaltyterm}
There are also cases in which the side constraint of \eqref{eq_altForm_allg} defines a non-convex feasible set. Then, the ADP problem is more difficult to solve. We give a simple two-dimensional example with the penalty term $R \colon \R^2 \to [0,\infty)$,
\begin{equation}
R(x_1, x_2) =
\begin{cases}
3 \mydot |x_1 - 5| & \text{for } 3 \mydot |x_1 - 5| \geqslant |x_2|,\\
|x_2|  & \text{for } |x_2| > 3 \mydot |x_1 - 5|.
\end{cases}    
\end{equation}
This functional has a non-centered minimum at $(5,0)^{\mathrm{T}}$ and the absolute value of its gradient $|\partial R(x)|$ is strongly dependent on the direction
. Because of these properties, it's easy to show that the term $\langle v,x \rangle$, $v \in \partial R(x)$ in the side constraint of \eqref{eq_altForm_allg} is non-convex w.r.t.\ $x \in \R^2$.
\end{remark}

%
%

\subsection{Parameter choice and early stopping} \label{ch_earlystopping}

By construction of the ADP model, we expect it in application to act like DIP.
But in the previous section it turned out that ADP behaves in fact equivalent to classical methods like Tikhonov's. When we apply ADP to an inverse problem, the question arises whether ADP can also deliver something that is ``new'' and not equivalent to a Tikhonov solution. This section presents, how the model has to be changed to produce ADP solutions that are more similar to DIP solutions. In the same time, we derive a strategy for choosing the parameter $\alpha$ of the ADP model.

When we compare the ADP method
\begin{equation} \label{eq_adp_alpha}
\begin{split}
    &\min_{B \in L(X,Y)} \frac{1}{2} \| A x(B) - y^\delta \|^2 \\
    \text{s.t.} \quad &x(B) = \mathrm{arg} \min_{x \in X} \frac{1}{2} \| B x - y^\delta \|^2 + \frac{\alpha_{\text{ADP}}}{2} \|x\|^2 
\end{split}
\end{equation}
to the equivalent (see Remark \ref{rem_l2_equivalence}) Tikhonov method
\begin{equation} \label{eq_tik_alpha}
    \min_{x \in X} \frac{1}{2} \| A x - y^\delta \|^2 + \frac{\tilde{\alpha}}{2} \|x\|^2,
\end{equation}
we have to make sure not to confuse the parameters $\alpha_{\text{ADP}}$ and $\tilde{\alpha}$ of both models with each other. At first we state the following relation between these parameters.

\begin{lemma} \label{lem_parameters}
If the solutions of \eqref{eq_adp_alpha} and \eqref{eq_tik_alpha} coincide, $\tilde{\alpha} \leqslant \alpha_{\text{ADP}}$ holds. Equality of the parameters could only occur if $y^\delta$ 
was in the kernel of $A^*$ or a singular vector of $A$.
\end{lemma}

The proof can be found in the appendix \ref{app_proof_parameters}. In general, we can assume that $\tilde{\alpha} < \alpha_{\text{ADP}}$ holds. So in any application it makes sense to choose the ADP parameter greater than one would choose the parameter of a Tikhonov model. But independent of the parameter choice, the ADP solution will always be equivalent to a Tikhonov solution (Remark \ref{rem_l2_equivalence}). To make ADP more similar to DIP, we apply early stopping \cite[section 7.8]{Goodfellow-et-al-2016}. This strategy is often used in the application of DIP but wasn't considered for the ADP model yet.

For a given inverse problem, we could solve the ADP problem \eqref{eq_adp_alpha} with a gradient descent algorithm w.r.t.\ the operator $B$ (see section \ref{ch_numeric} for details) and terminate this iteration early. Taking $B^0 = A$ as initial value leads by definition to $x(B^0)$ being equal to the Tikhonov solution w.r.t.\ the parameter $\alpha_{\text{ADP}}$. We assume the iteration to converge successfully towards the minimizer $\hat{x}$ of \eqref{eq_adp_alpha}. Since $\hat{x}$ is also the minimizer of \eqref{eq_tik_alpha}, the limit of the iteration is also a Tikhonov solution but w.r.t.\ the parameter $\tilde{\alpha}$. Because of $\tilde{\alpha} < \alpha_{\text{ADP}}$, the starting solution $x(B^0)$ is a stronger regularized Tikhonov solution than the limit $\hat{x}$ of the iteration (see figure \ref{fig_earlystopping}).

\begin{figure}[ht]
    \centering
    \includegraphics[scale=1]{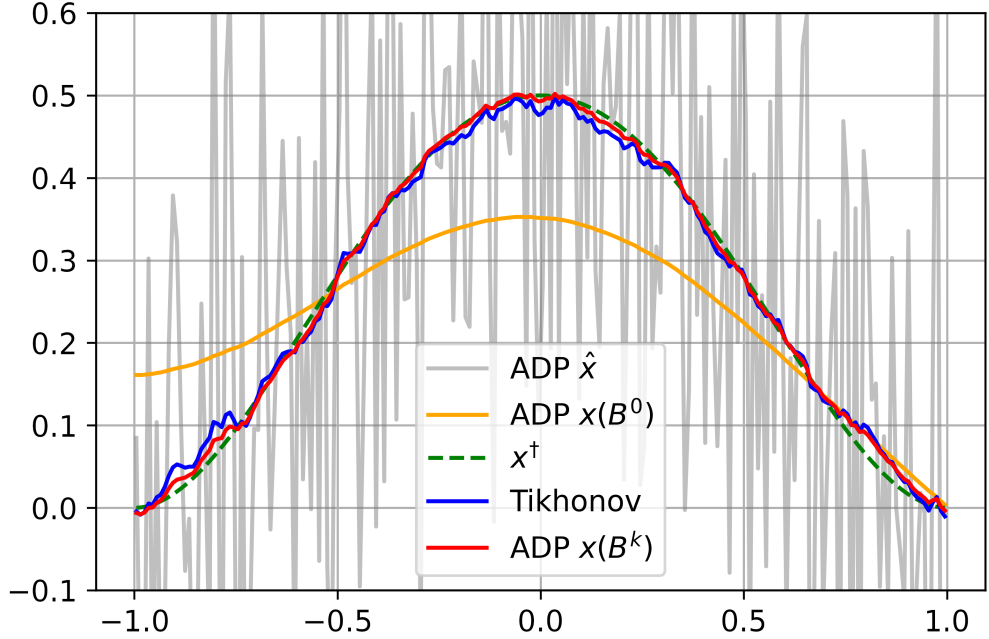}
    \caption{Comparison of ADP solutions during gradient descent to the Tikhonov method
    (orange: at the start of the gradient descent, red: by early stopping, gray: limit of the gradient descent).
    The forward operator is an integration like in \eqref{eq_forwardoperators} and the data $y^\delta$ is contaminated with gaussion noise (PSNR=40).
    The Tikhonov parameter and the stopping criterion for ADP were chosen a posteriori to achieve the most accurate reconstructions.}
    \label{fig_earlystopping}
\end{figure}

If we apply early stopping, we take some $x(B^k)$ in between, which in general does not equal a Tikhonov solution w.r.t.\ $A$ (see figure \ref{fig_earlystopping}).
This strategy makes sense if we expect $x(B^k)$ to be a better solution than (the Tikhonov solutions) $x(B^0)$ and $\hat{x}$. That could be the case if $x(B^0)$ is slightly over-regularized and $\hat{x}$ is slightly under-regularized. Because then, the optimal regularization would lay in between.

Now, we come back to the parameter choice. If we have a criterion for estimating a suitable Tikhonov parameter $\alpha_{\text{Tik}}$ for a given inverse problem, we should try to choose $\alpha_{\text{ADP}}$ in a way that
\begin{equation}
    \tilde{\alpha} < \alpha_{\text{Tik}} < \alpha_{\text{ADP}}
\end{equation}
holds. Because then, $x(B^0)$ will be slightly over-regularized and $\hat{x}$ slightly under-regularized, as proposed.

In the example of figure \ref{fig_earlystopping}, we see that the ADP solution $x(B^k)$, obtained with early stopping, is a better approximation for the ground truth $x^\dagger$ than the most accurate Tikhonov solution, which corresponds to $\alpha_{\text{Tik}}$. But this result is strongly dependent on the particular inverse problem. The Tikhonov method is optimal for data that is normally distributed. If the given distribution differs from that, it is theoretically possible that ADP with early stopping produces a better solution than the Tikhonov method.

Finally, we want to include the early stopping strategy directly into the ADP model to be able to investigate its effect on the regularization of inverse problems. Early stopping enforces the iterated variable to stay close to the initial value. 
Because of $B^0=A$, we can expect $\|B^k - A\|$ to be small for small $k$. This leads to using $\|B^k - A\|$ as an additional penalty term in the ADP problem, which has a similar effect as early stopping \cite[section 2.3]{bishop_1995}, \cite[section 4]{sjoberg1994}. What comes out is the ADP-$\beta$ model
\begin{equation}
\begin{split} 
    \min_{B \in L(X,Y)} \frac{1}{2} \| A x(B) - y^\delta\|^2  + \beta \|B-A\|^2\\
    \text{s.t.} \quad x(B) =  \mathrm{arg} \min_{x \in X} \, \frac{1}{2} \|Bx - y^\delta\|^2 + \alpha R(x).
\end{split}
\end{equation}

%
%

\subsection{Properties of ADP}

The equivalence between ADP and the Ivanov method (with general convex penalty term $R$), shown in section \ref{ch_equivalence}, allows to obtain some regularization properties (existence, stability, convergence) for ADP.
We suppose that Assumption \ref{def_assumptions} holds. Besides the functional
\begin{equation} \label{eq_tildeR}
    \tilde{R}(x) = \min_{v \in \partial R(x)} \langle v, x \rangle
\end{equation}
is assumed to be well-defined. Because then, the side constraint of \eqref{eq_altForm_allg} can be formulated as
\begin{equation} \label{eq_sideconstraint}
    \tilde{R}(x) \leqslant \frac{\|y^\delta\|^2}{4 \alpha}.
\end{equation}
Due to $\langle v, x \rangle = R(x) + R^*(v)$ for $v \in \partial R(x)$, where $R^*$ denotes the convex conjugated functional, coercivity of $R$ implies coercivity of $\tilde{R}$.

\begin{remark}[Existence]
There exists a solution of the ADP problem \eqref{eq_ADP} if the functional $\tilde{R}$, defined in \eqref{eq_tildeR}, is weakly lower semicontinuous. This follows from the equivalence theorem~\ref{theo_equivalence} and \cite[Theorem 2.1]{Vogel1990}
about the existence of Ivanov solutions.
\end{remark}

Uniqueness of solutions and stability w.r.t.\ the data $y^\delta$ is less trivial. First, the right hand side of the side constraint \eqref{eq_sideconstraint} is dependent on $y^\delta$, which isn't the case for ordinary Ivanov problems. Secondly, we know from Remark \ref{rem_bad_penaltyterm}, that the constraint \eqref{eq_sideconstraint} does not always define a convex feasible set. Nevertheless, for the special case $R(x) = \frac{1}{2}\|x\|^2$ we can obtain a convenient stability result. In this case, $\tilde{R}(x) = \|x\|^2$ is a strictly convex functional.
Additionally, if the given inverse problem is ill-posed, we can assume the ADP solutions to fulfill the constraint \eqref{eq_sideconstraint} with equality. Under these conditions, the following theorem provides stability of ADP.

\begin{theorem}[Stability] \label{theo_stabilityADP}
For $R(x) = \frac{1}{2}\|x\|^2$, let $(y_k) \subset Y$ be a sequence with ${y_k \to \hat{y} \in Y}$ and assume that the corresponding ADP solutions $x_k$, $\hat{x}$ are unique and fulfill the side constraint in (\ref{eq_altFormtik}) with equality. Then ADP is stable, which means $x_k \to \hat{x}$.
\end{theorem}

The proof can be found in the appendix \ref{app_proof_stabilityADP}.

To obtain a convergence result for ADP, it makes sense to use standard convergence theorems, either of the Tikhonov method \cite[Theorem 4.4]{Hofmann_2007} or of the Ivanov method \cite[Theorem 2.5]{Kaltenbacher_2018}, \cite[Theorem 3]{Seidman_1989}. They differ especially in the source conditions they require for the ground truth $x^\dagger$ and in the parameter choice rules. If we assume $\tilde{R}$ to be convex, by \cite[Theorem 2]{Vasin1970RelationshipOS} and the equivalence theorem \ref{theo_equivalence}, the Tikhonov problem
\begin{equation} \label{eq_tikhonovtilde}
    \min_{x \in X} \frac{1}{2} \|A x - y^\delta\|^2 + \tilde{\alpha} \tilde{R}(x)
\end{equation}
is equivalent to the ADP formulation \eqref{eq_altForm_allg} for suitable chosen $\tilde{\alpha} \geqslant 0$.

\begin{remark}[Convergence]
Because of the equivalence between \eqref{eq_altForm_allg} and \eqref{eq_tikhonovtilde}, the convergence of ADP solutions $x_\alpha^\delta$ to $x^\dagger$ for vanishing $\delta$ w.r.t.\ the Bregman distance can be directly derived from Tikhonov convergence theorems. But the ADP parameter $\alpha$ does not coincide with the Tikhonov parameter $\tilde{\alpha}$. That's why, for ADP we do not get an explicit parameter choice rule like $\alpha \sim \delta$. Besides, a source condition for $x^\dagger$ has to be fulfilled by the functional $\tilde{R}$ (defined in \eqref{eq_tildeR}) and not by the penalty term $R$.
\end{remark}

%
%

\subsection{Properties of ADP-\textbeta}

For proving the existence of solutions of variational regularization schemes, \cite[Theorem 3.1]{Hofmann_2007} provides a useful framework. If we want to apply this for ADP-$\beta$,
it has to be ensured that $B \mapsto x(B)$ is weak-weak continuous \cite[Assumptions 2.1]{Hofmann_2007}. But unfortunately, in general this is not the case. 

To obtain convenient regularization properties anyway, we restrict to ${X=Y=L^2(\Omega)}$ with $\Omega \subset \R^n$. In this setting, we consider a forward operator $A \colon X \to Y$ that can be parametrized by a function $f \in L^p(\Omega)$, $p \in [1,\infty)$. More precisely, we take a continuous, bilinear operator ${T \colon L^p(\Omega) \times X \to Y}$ and define
\begin{equation}
    Ax = T(f,x).
\end{equation}
The same parametrization of operators by functions is used in \cite{bleyer2013}. One typical example would be a convolutional operator $T(f,x) = f*x$. 

The crucial idea is the additional restriction $f \in W^{1,p}(\Omega)$ to take advantage of the compact embedding of Sobolev spaces $W^{1,p}(\Omega) \subset L^p(\Omega)$. A similar strategy is used in \cite{Kluth_2020} for achieving weak-weak continuity of the forward operator.

We define the parametrized ADP-$\beta$ approach as 
\begin{align} \label{eq_ADP-param}
\begin{split}
\min_{g \in W^{1,p}} \frac{1}{2} \| T(f,x_g) - y^\delta\|_{L^2}^2  &+ \beta\|f-g\|_{W^{1,p}}^2 \\
\text{s.t.} \quad x_g = \mathrm{arg} \min_{x \in L^2} \, \frac{1}{2} \|T(g,x) - y^\delta\|_{L^2}^2 &+ \alpha R(x).
\end{split}
\end{align}
In particular, this can be interpreted as a Tikhonov method for solving the nonlinear inverse problem $F(g^\dagger) = y^\dagger$ with the forward operator $F \colon W^{1,p}(\Omega) \to Y$, $F(g) = T(f, x_g)$.

\begin{remark}[Existence] \label{rem_betaexistence}
The forward operator $F$ is weak-strong continuous if the penalty term $R$ is strongly convex. This holds, because weak convergence $g_k \rightharpoonup g$ w.r.t.\ $W^{1,p}(\Omega)$ implies convergence by norm in $L^p(\Omega)$, by Theorem \ref{theo_xBstetig} the convergence of $x_{g_k} \to x_g$ follows, and the bilinear operator $T$ is continuous. This is more than enough to fulfill the assumptions of \cite[Theorem 3.1]{Hofmann_2007}, which provides the existence of a solution of \eqref{eq_ADP-param}.
\end{remark}

A weak stability result for the parametrized ADP-$\beta$ method could be directly obtained from \cite[Theorem 3.2]{Hofmann_2007}. But this particular framework even allows to prove strong stability.

\begin{theorem}[Stability] \label{theo_stabilitybeta}
For $p=2$, let $R$ be a strongly convex penalty term, $(y_k) \subset Y$ a convergent sequence with $y_k \to \hat{y}$ and $(g_k) \subset W^{1,p}(\Omega)$ the corresponding solutions of the ADP\nobreakdash-$\beta$ problem \eqref{eq_ADP-param}. Then, $(g_k)$ has a convergent subsequence and the limit of each subsequence is an ADP-$\beta$ solution corresponding to $\hat{y}$.
\end{theorem}

The proof can be found in the appendix \ref{app_proof_stabilitybeta}.

While proving existence and stability of ADP-$\beta$-solutions required a smart parametrization and the use of compact embeddings, a convergence theorem (w.r.t.\ the Bregman distance) can be proven for the general formulation \eqref{eq_ADPbeta}. Similar to classical results like \cite[Theorem~4.4]{Hofmann_2007} or \cite[Theorem~2]{Burger_2004}, we need to assume a source condition 
\begin{equation} \label{eq_sourcecondition}
    \exists w \in Y: \quad A^*w \in \partial R (x^\dagger).
\end{equation}
The parameter $\beta$ turns out to be really helpful for obtaining a convergence result.

\begin{theorem}[Convergence] \label{theo_convergencebeta}
Let Assumption \ref{def_assumptions} hold, $x^\dagger$ be an $R$-minimizing solution of \eqref{eq_inverseproblem} which fulfills the source condition \eqref{eq_sourcecondition} and assume there exist ADP-$\beta$ solutions $\hat{x}_\alpha^\delta$ of \eqref{eq_ADPbeta}. If $\alpha$ is chosen proportional to $\delta$, there exists $d \in D_R(\hat{x}_\alpha^\delta, x^\dagger)$ which fulfills $d = O(\delta)$.
\end{theorem}

The proof can be found in the appendix \ref{app_proof_convergencebeta}.

%
%

\section{Numerical Computations} \label{ch_numeric}

\textbf{Aim.}
We want to see whether there is a similarity between ADP and DIP also on the numerical side.
The ADP approach is based on the idea of using a LISTA network in a DIP method. Usually LISTA architectures contain round about ten layers, but ADP is motivated with a network of infinite depth (see section \ref{ch_prelim}). To derive the ADP model, the output of this infinite network is then replaced by the solution of a minimization problem. So the question arises, whether numerically computed ADP solutions of an inverse problem are yet similar to solutions obtained via a DIP with LISTA architecture.

In this section, we present algorithms for the computation of ADP solutions and we compare them with DIP solutions. In doing so, the focus is not on the performance of the methods (in comparison to other state-of-the-art reconstruction algorithms) but on the similarity of the different solutions.

\noindent \textbf{Methods.}
From Theorem \ref{theo_equivalence}, we know that the ADP problem is equivalent to an Ivanov problem. This creates a possibility to compute ADP solutions easily, fast and almost exactly (we call this method ADP Ivanov). In contrast to that, it is more difficult to realize a LISTA architecture with infinite depth. But there are at least two possibilities to simulate such a network.

The first idea (Algorithm \ref{algo_lista_infty}: DIP LISTA $L=\infty$) is to begin with a network $\varphi_B$ of ten layers and to increase the network depth during the training process of the DIP. This is done implicitly with a simple trick. In each training step, the network's input is set to be the network's output of the previous step \cite[Appendix 3]{otero2020}. So the original input will pass through more and more layers and in each step the last ten layers are optimized (via backpropagation).

\begin{algorithm}[h]
\SetAlgoLined 
 initialize $B_0$, $z_0$ (e.g.\ $B_0 = A$, $z_0 = \text{random noise}$)\;
 \For{$k=0,1,...$}{
  $z_{k+1} = \varphi_{B_k}(z_k)$\;
  $\mathrm{loss}_k = \frac{1}{2}\|A \varphi_{B_k}(z_k) - y^\delta\|^2$\;
  $B_{k+1} =\mathrm{update}(\nabla_{B_k} \mathrm{loss}_k)$\;
 }
 \Return $z_k$
 \caption{DIP LISTA L=$\infty$}
 \label{algo_lista_infty}
\end{algorithm}

\begin{algorithm}[h]
\SetAlgoLined 
 initialize $B_0$ (e.g.\ $B_0 = A$)\;
 \For{$k=0,1,...$}{
  1. Calculate $x(B_k)$ with fixed point iteration\\
  2. IFT provides: $\nabla_{B_k} x(B_k)$\;
  3. Update:\\
  $\mathrm{loss}_k = \frac{1}{2}\|A x(B_k) - y^\delta\|^2$\;
  $B_{k+1} = \mathrm{update}(\nabla_{B_k} \mathrm{loss}_k) $\; 
 }
 \Return $x(B_k)$
 \caption{ADP IFT}
 \label{algo_ift}
\end{algorithm}

The second idea (Algorithm \ref{algo_ift}: ADP IFT) is to 
compute $x(B)$ from \eqref{eq_ADP} with a classical algorithm like ISTA. After that, one can compute the gradient of $x(B)$ w.r.t.\ $B$ (see the proof of \cite[Lemma 4.1]{otero2020}) via the implicit function theorem (IFT). Thus, backpropagation through a big amount of layers is avoided.

For the standard DIP approach, we use a LISTA-like architecture of ten Layers (DIP LISTA $L=10$) and optimize the weights via backpropagation. So, in total we compare four different methods (ADP Ivanov, ADP IFT, DIP LISTA $L=\infty$, DIP LISTA, $L=10$). Since solving the Ivanov problem results in the exact ADP solution, we use this as a reference for the other three methods (for which we don't have convergence guarantees).

In all methods we use the elastic net functional \cite{zou_hastie2005} $R(x) = \alpha_1 \|x\|_1 + \frac{\alpha_2}{2} \|x\|^2$ 
as a penalty term. So there is one parameter for $\ell^1$-regularization (leads to sparsity) and one parameter for $\ell^2$-regularization (leads to stability and smoothness). In the LISTA-architecture, this is realized by substracting the gradient of the $\ell^2$-term before applying the activation function.

\noindent \textbf{Setting.}
We consider two different artificial inverse problems (inversion of the integration operator and a deconvolution)
on $L^2(I)$ for an interval $I \subset \R$. The forward operators are
\begin{equation}
(A_{1} x)(t) = \int_0^t x(s) \, \mathrm{d}s \quad \text{ and } \quad A_{2}x = g * x, \label{eq_forwardoperators}
\end{equation}
$g$ being a Gaussian function. Both of them lead to ill-posed inverse problems. We chose three different ground truth functions and created data by applying the forward operators and adding normally distributed random noise.
This leads to six examples in total, which is enough for some basic observations. 
Figure \ref{fig_numericways_int} shows the reconstructions corresponding to the integration operator $A_1$. The three rows contain the three different ground truth functions and each column contains a different method. For comparison, the actual ADP solution (ADP Ivanov) and the ground truth is displayed in every plot.
Since we are only interested in finding similarities and disparities between the solutions of the different methods, the choice of the regularization parameters plays a minor role. So, we took the same values $\alpha_1$, $\alpha_2$ for each method and simply chose them a posteriori for each example to minimize the $L^2$-error between reconstructions and ground truth. Figure \ref{fig_numericways_conv} shows the analogous results for the deconvolution problem (forward operator $A_2$).

\begin{figure}[t!]
\centering
\includegraphics[width=0.32\textwidth]{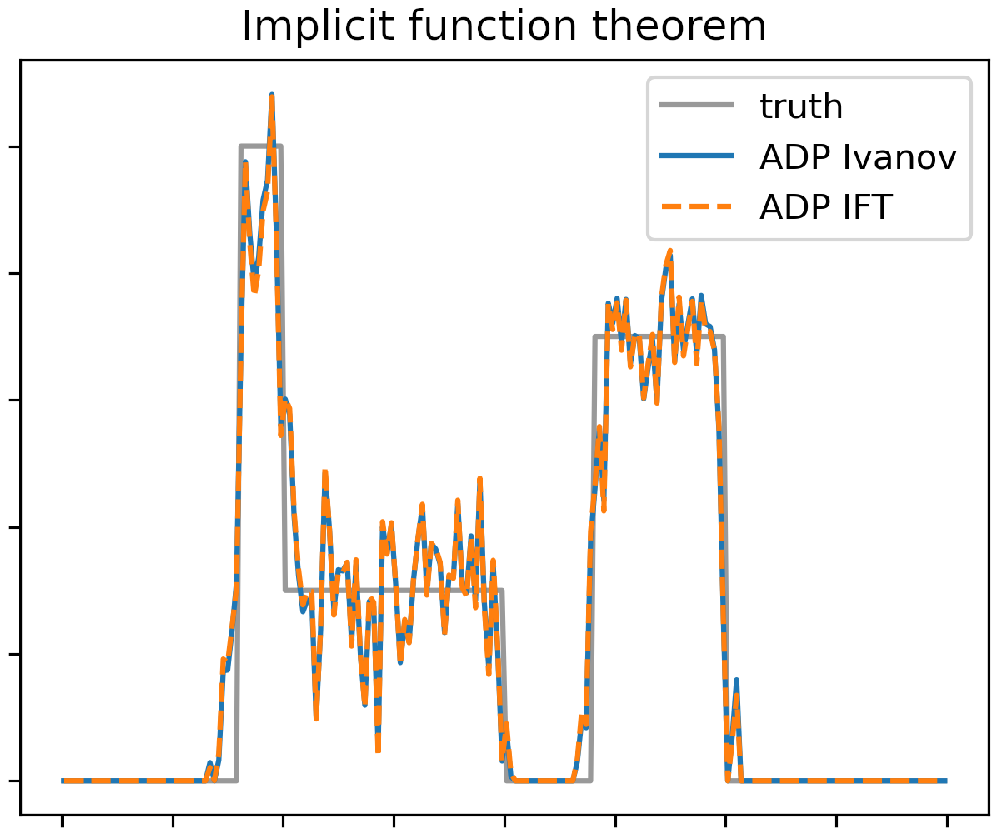}
\includegraphics[width=0.32\textwidth]{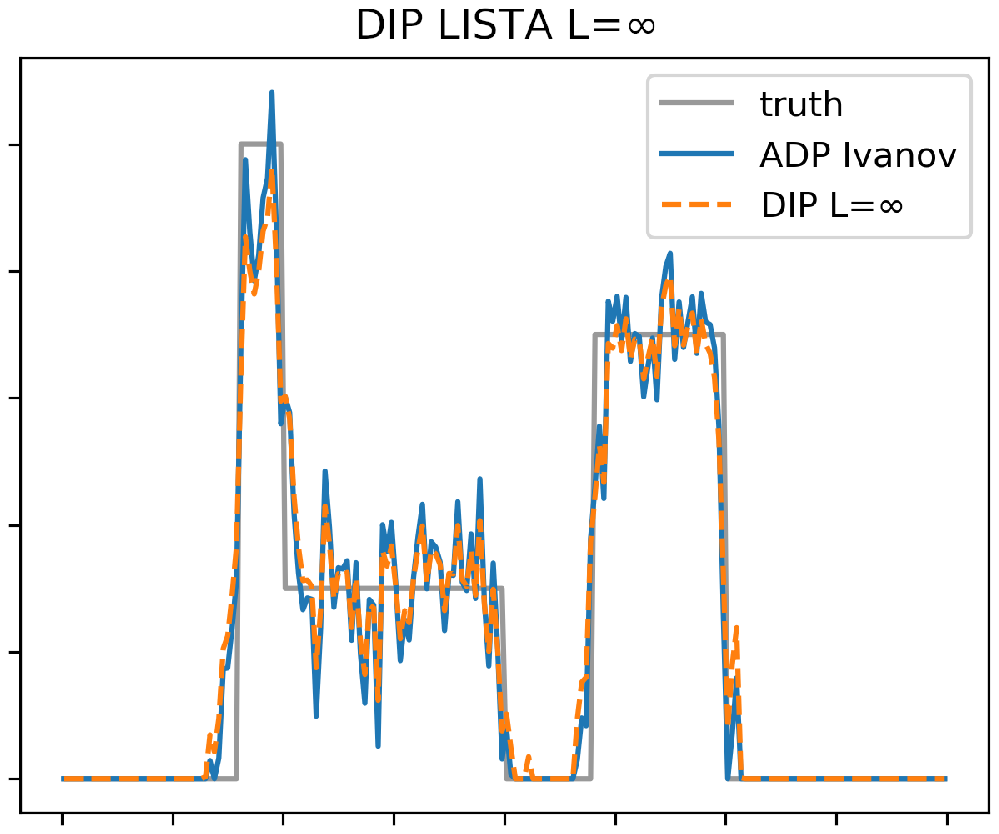}
\includegraphics[width=0.32\textwidth]{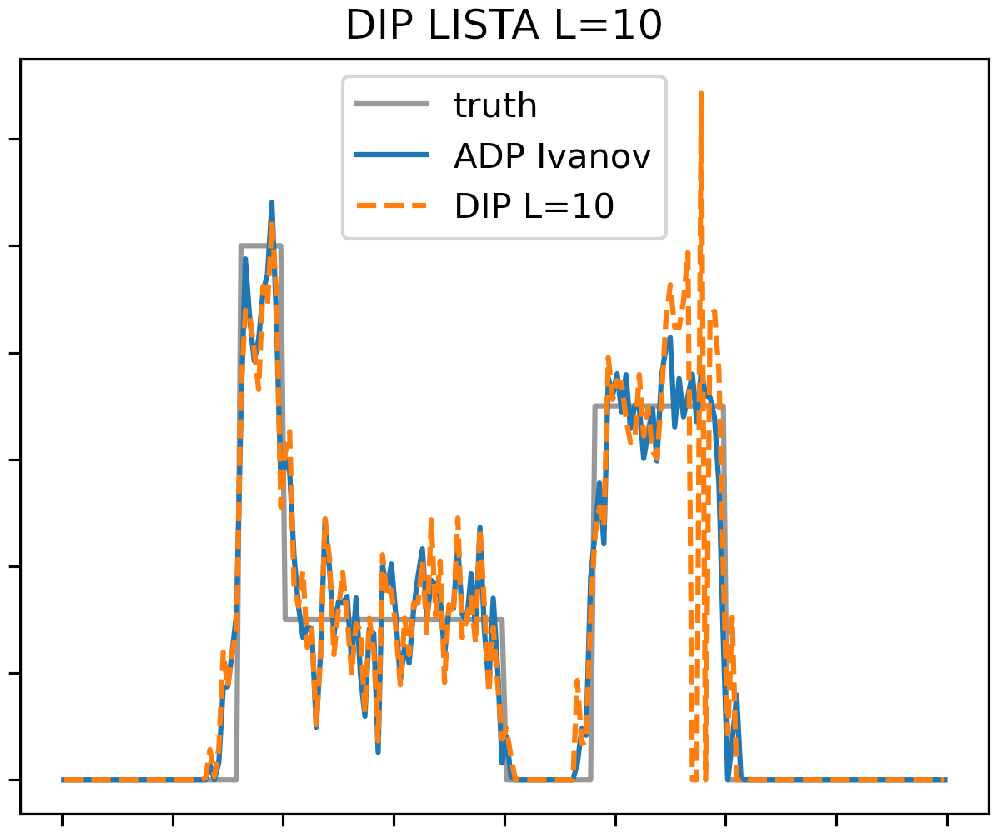}

\includegraphics[width=0.32\textwidth]{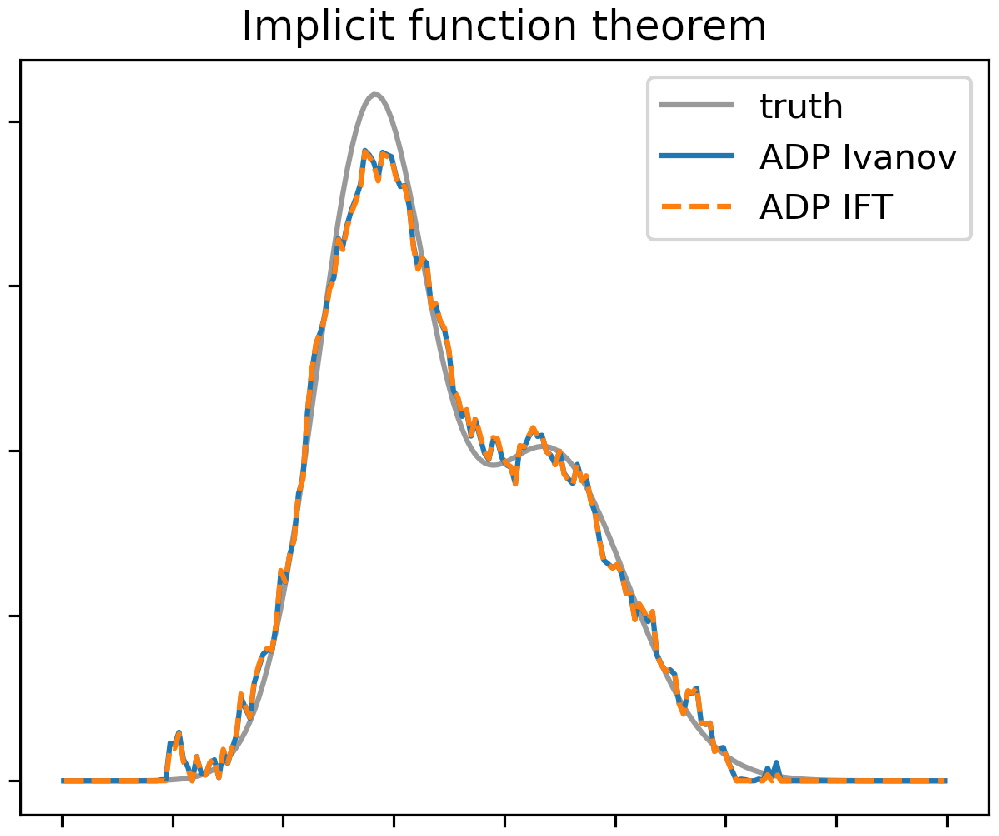}
\includegraphics[width=0.32\textwidth]{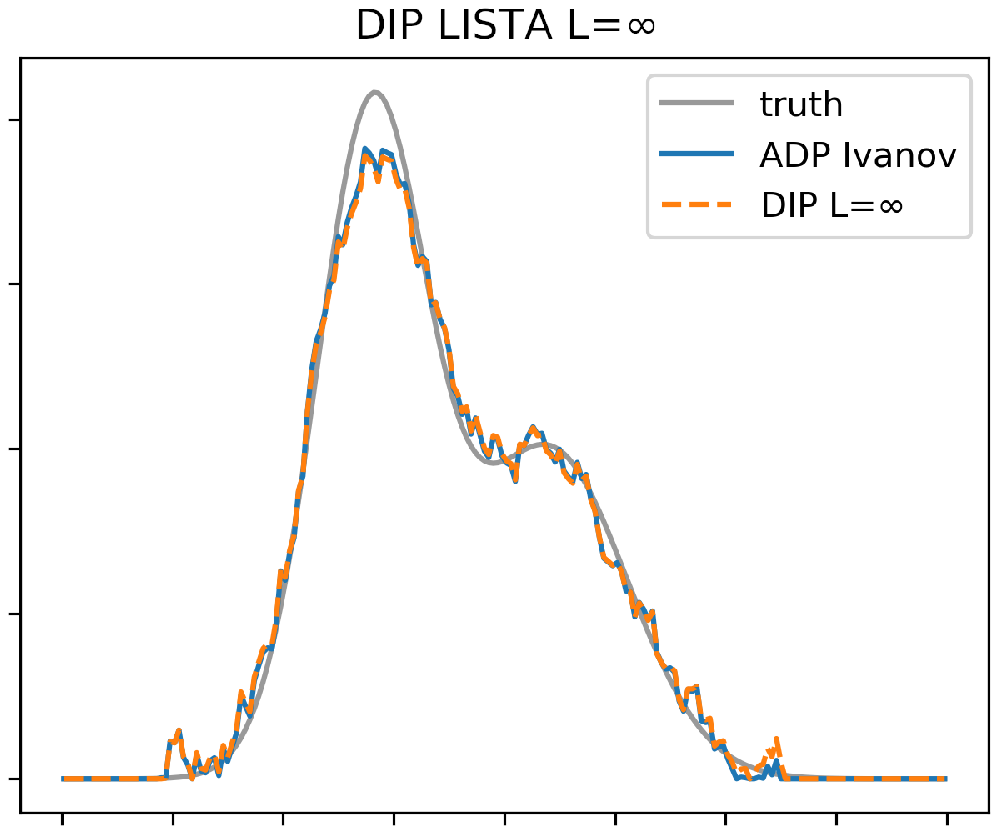}
\includegraphics[width=0.32\textwidth]{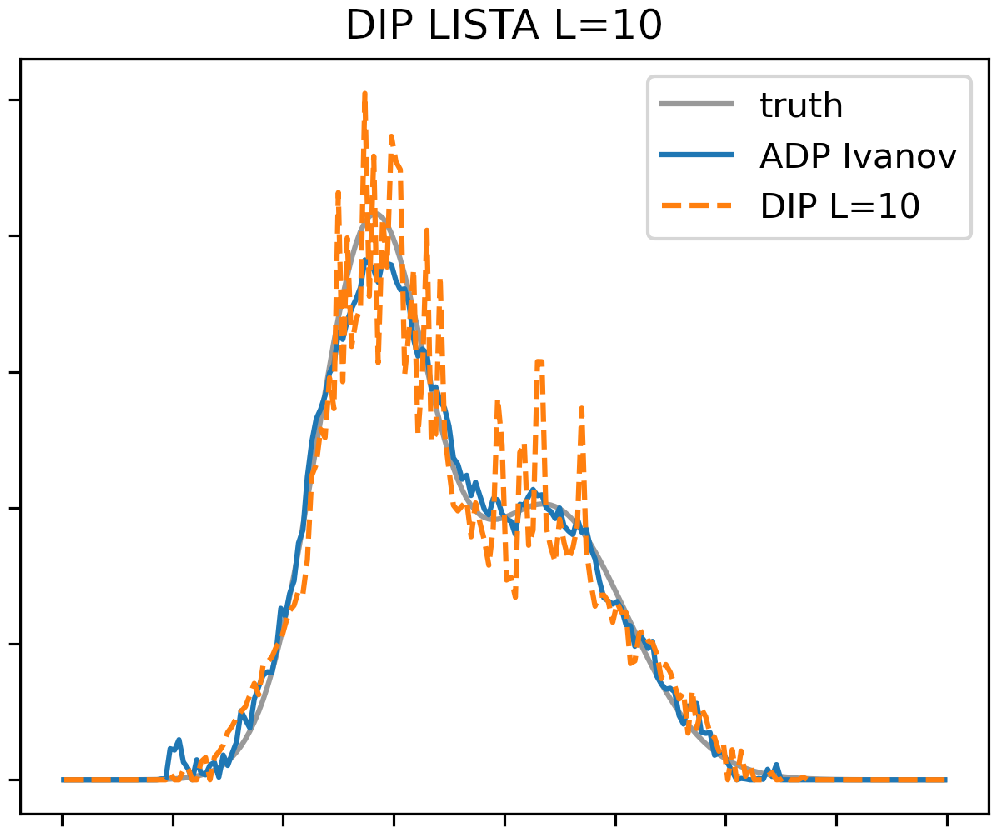}

\includegraphics[width=0.32\textwidth]{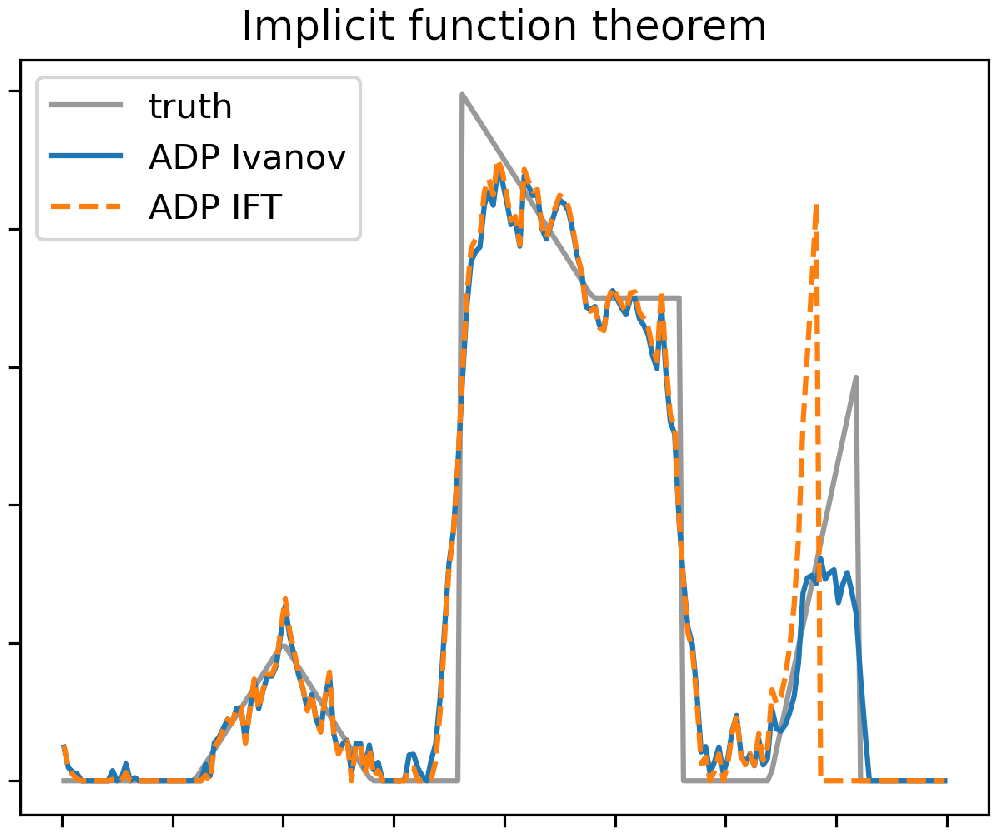}
\includegraphics[width=0.32\textwidth]{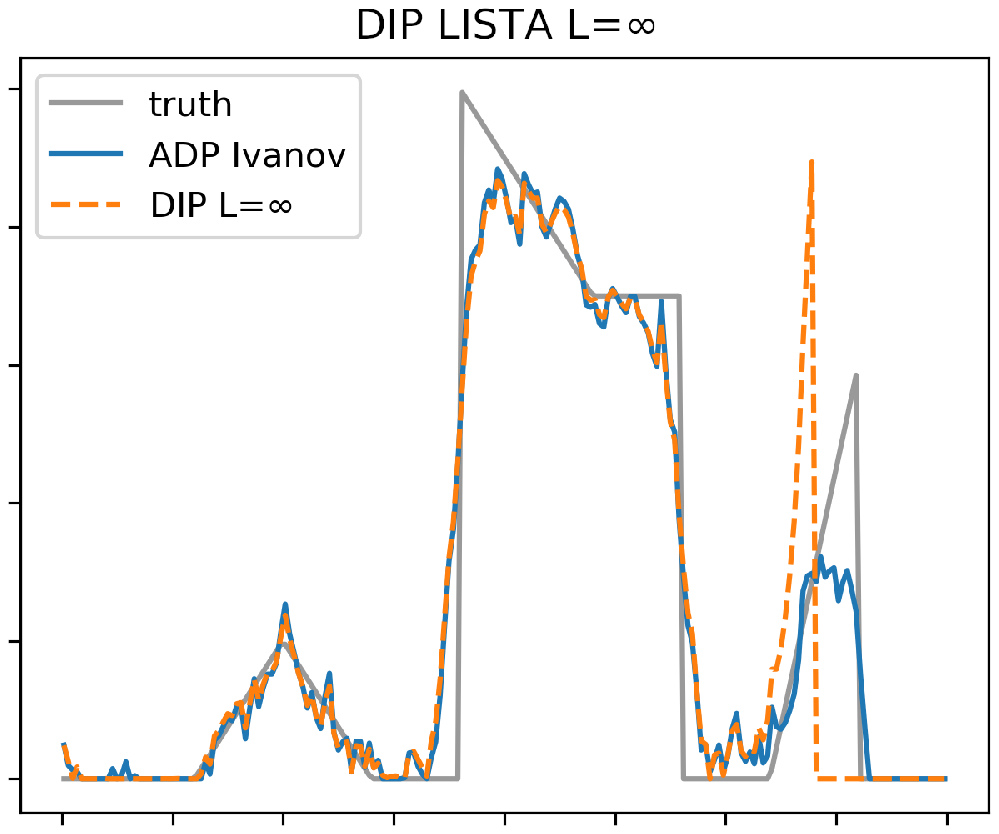}
\includegraphics[width=0.32\textwidth]{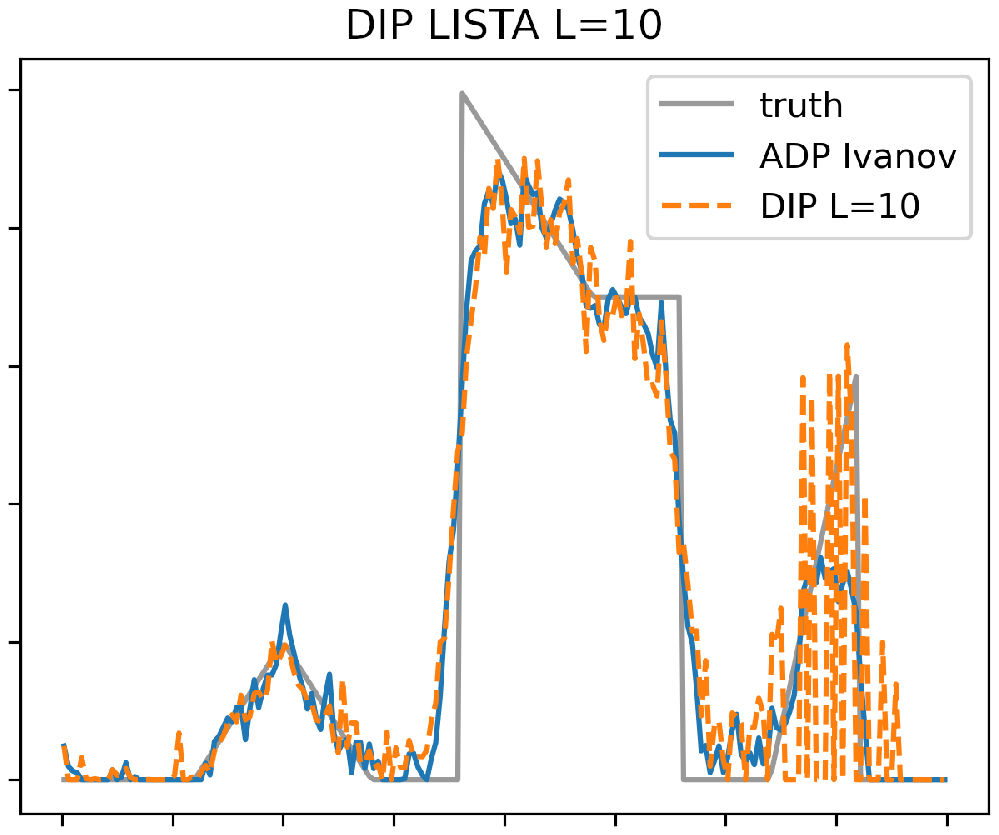}
\caption{Computation of ADP and DIP reconstructions via the IFT, via a DIP with LISTA network ($L=\infty$ and $L=10$) and via the equivalent Ivanov problem. The forward operator is $A_1$ (integration). The given data has PSNR=40 due to additiv Gaussian noise. The regularization parameters $\alpha_1$, $\alpha_2$ are chosen for each example (row) separately but are the same for each method (column).}
\label{fig_numericways_int}
\end{figure}

\noindent \textbf{Observations.}
From these experiments, we can make the following observations.
There is a significant difference between using $L=10$ or $L=\infty$ layers in a LISTA network. With an infinite number of layers, the reconstructions are looking more realistic.
The results of the DIP LISTA L=$\infty$ (Algorithm \ref{algo_lista_infty}) method and of the IFT method (Algorithm \ref{algo_ift}) are always looking quite similar. This was expected because both of them simulate an infinitely deep LISTA network. Differences are probably due to the different ways the gradients are computed or due to to slow convergence of the methods.

\begin{figure}[t!]
\centering
\includegraphics[width=0.32\textwidth]{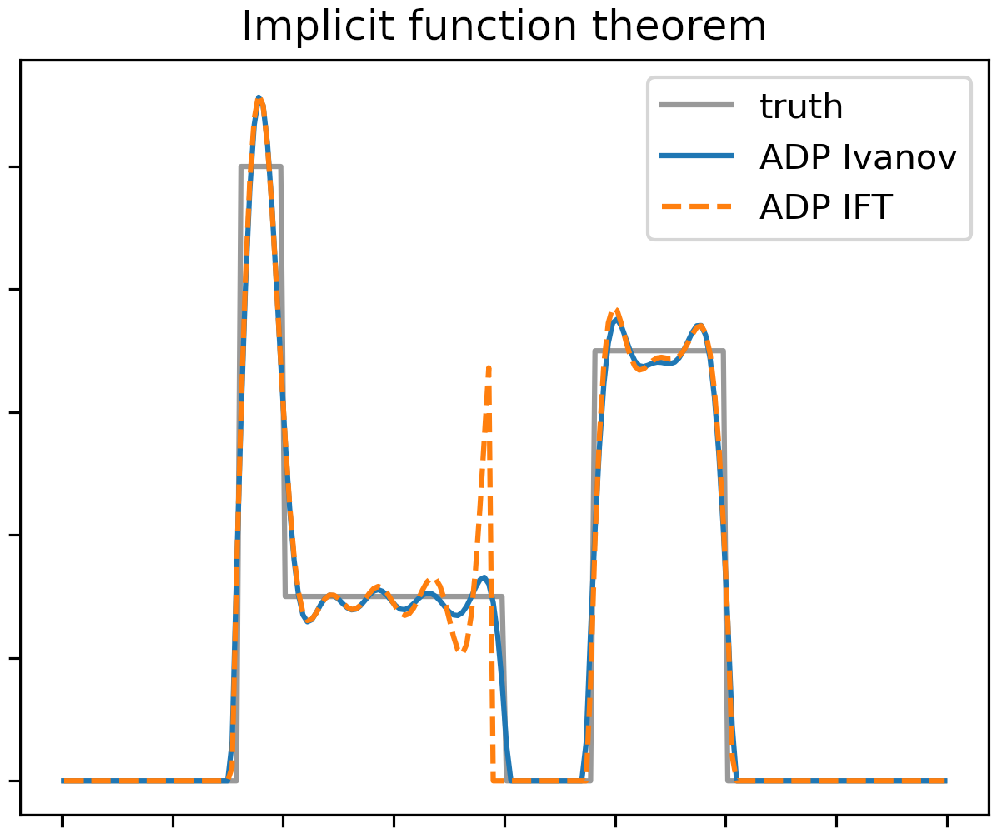}
\includegraphics[width=0.32\textwidth]{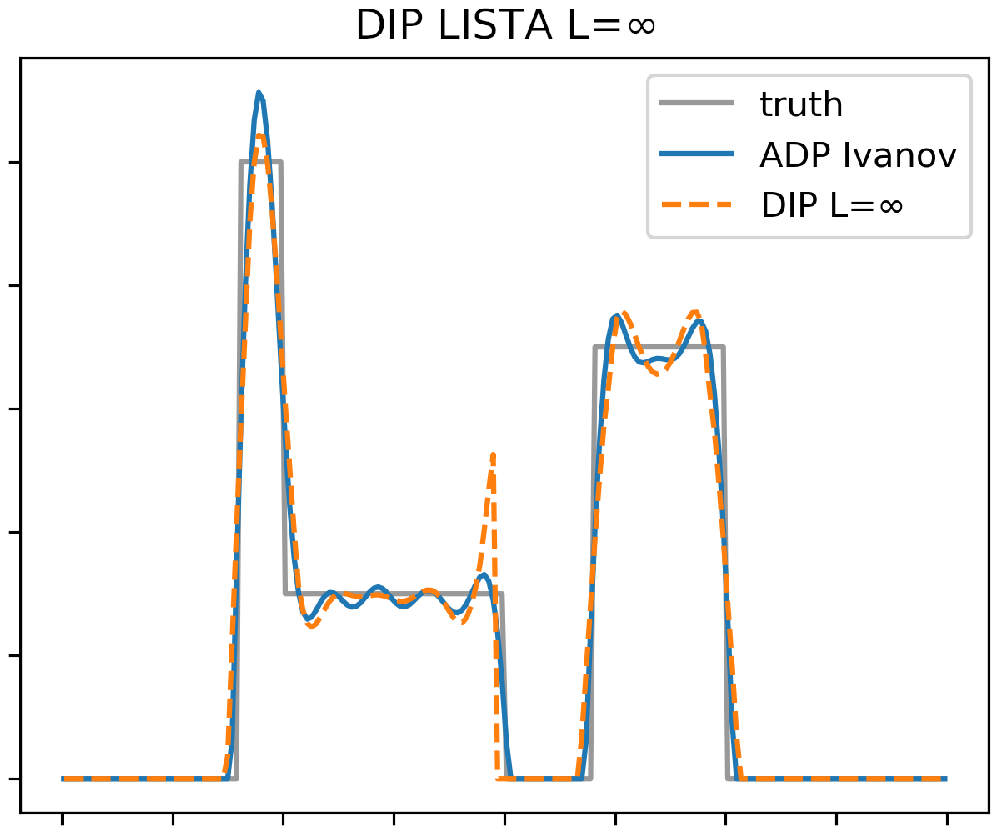}
\includegraphics[width=0.32\textwidth]{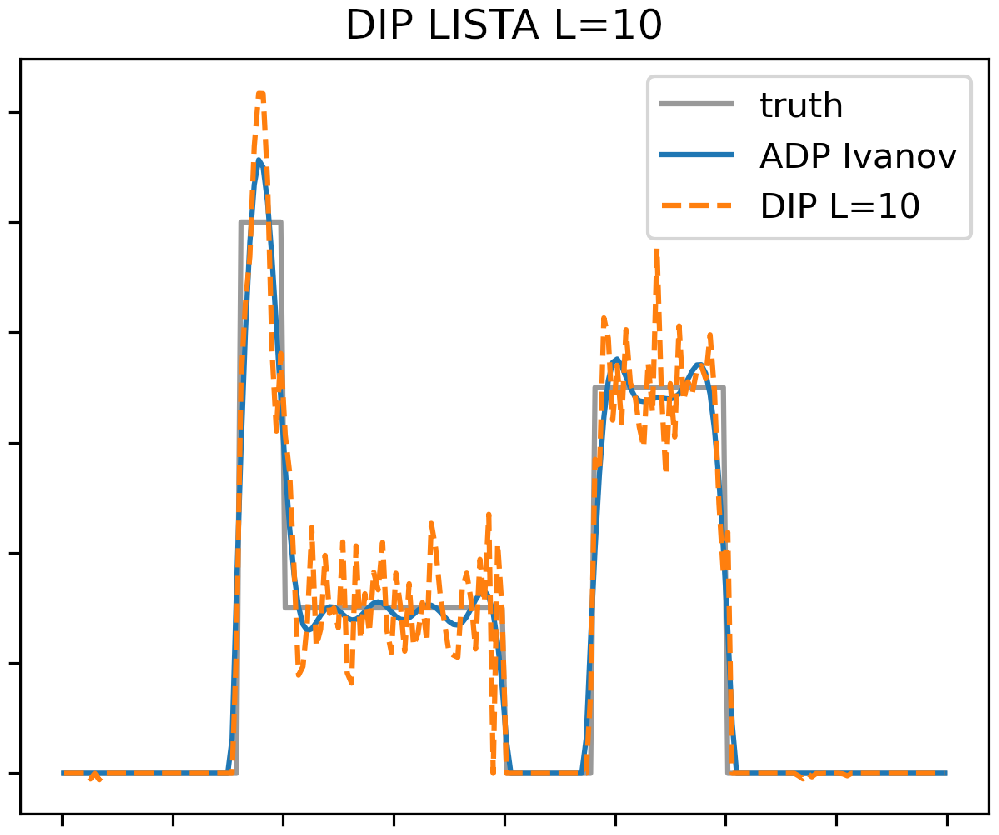}

\includegraphics[width=0.32\textwidth]{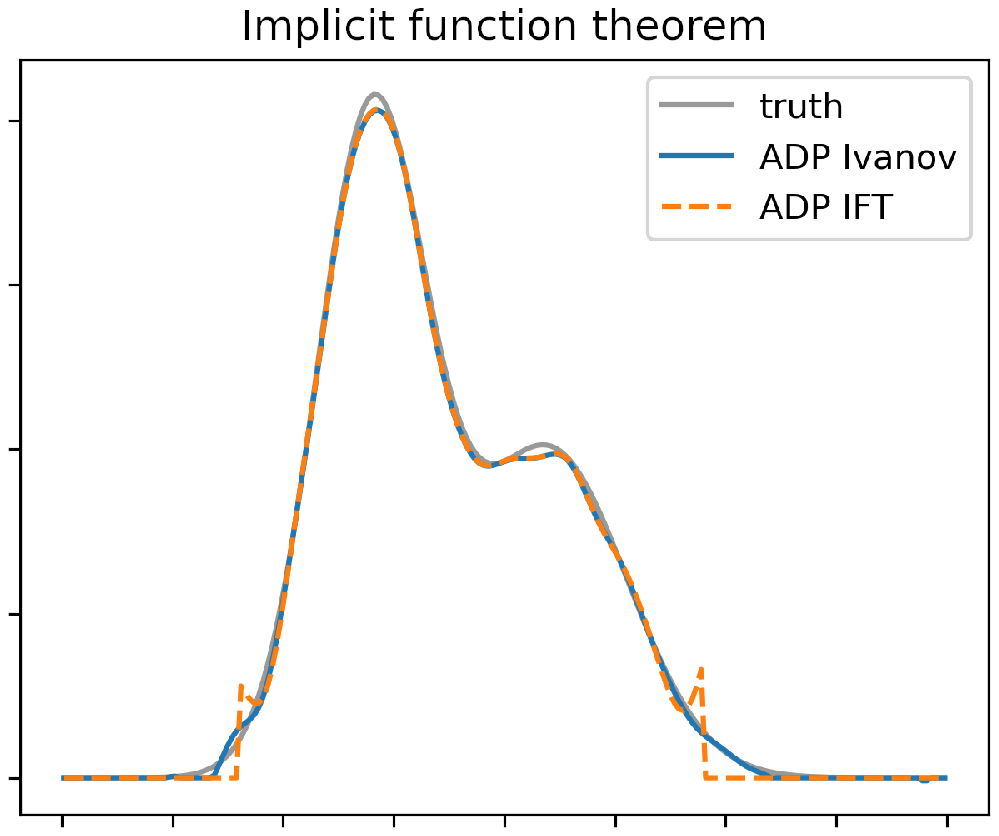}
\includegraphics[width=0.32\textwidth]{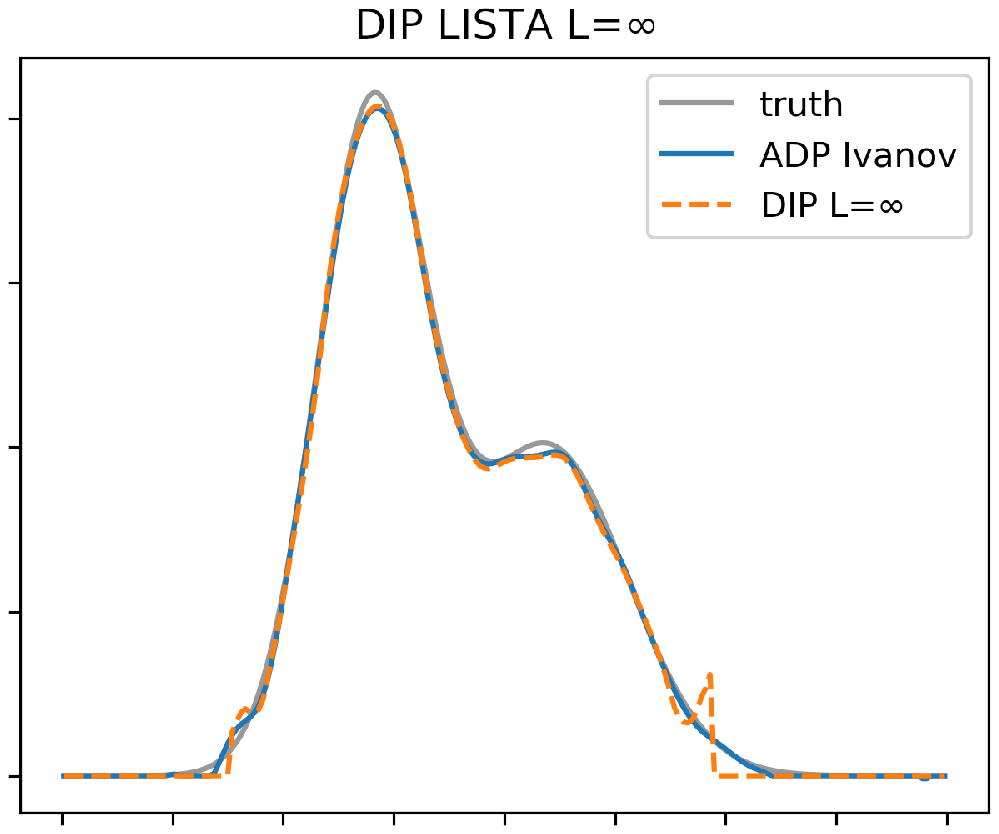}
\includegraphics[width=0.32\textwidth]{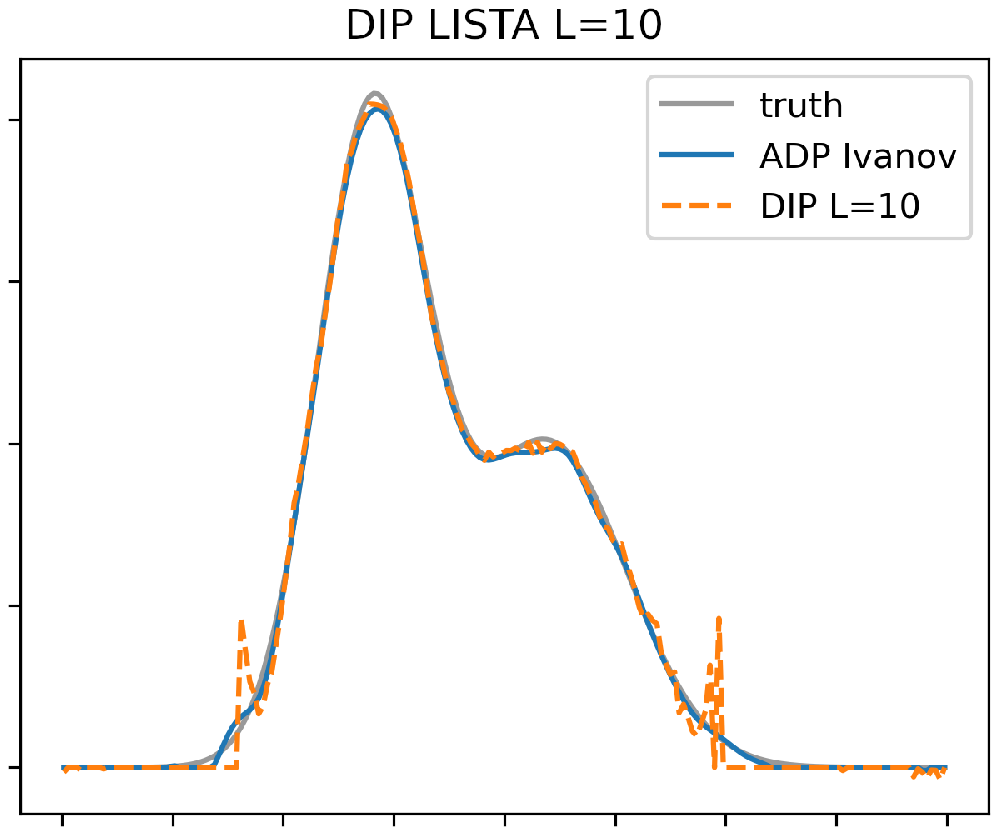}

\includegraphics[width=0.32\textwidth]{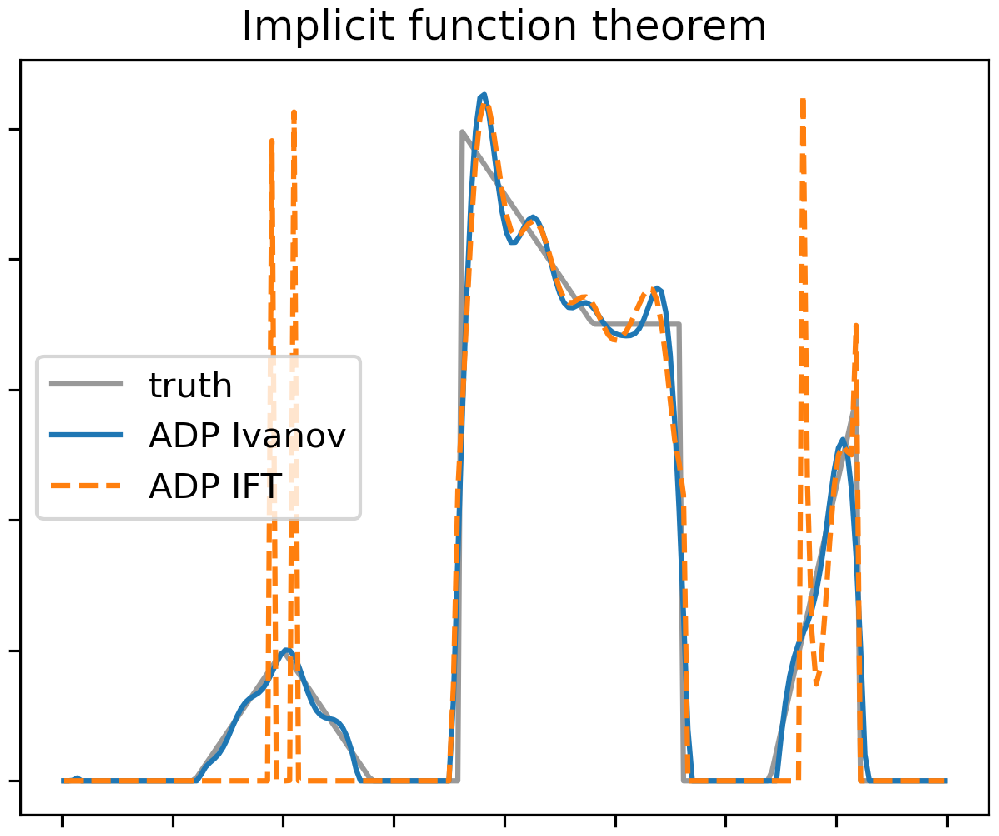}
\includegraphics[width=0.32\textwidth]{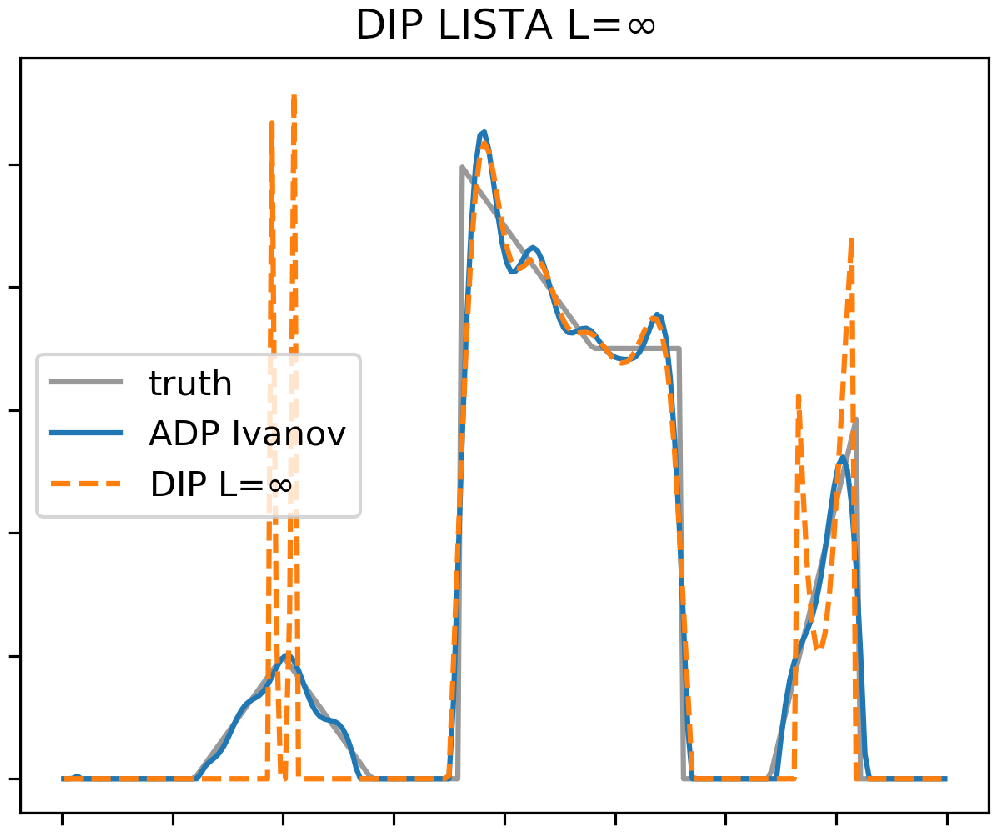}
\includegraphics[width=0.32\textwidth]{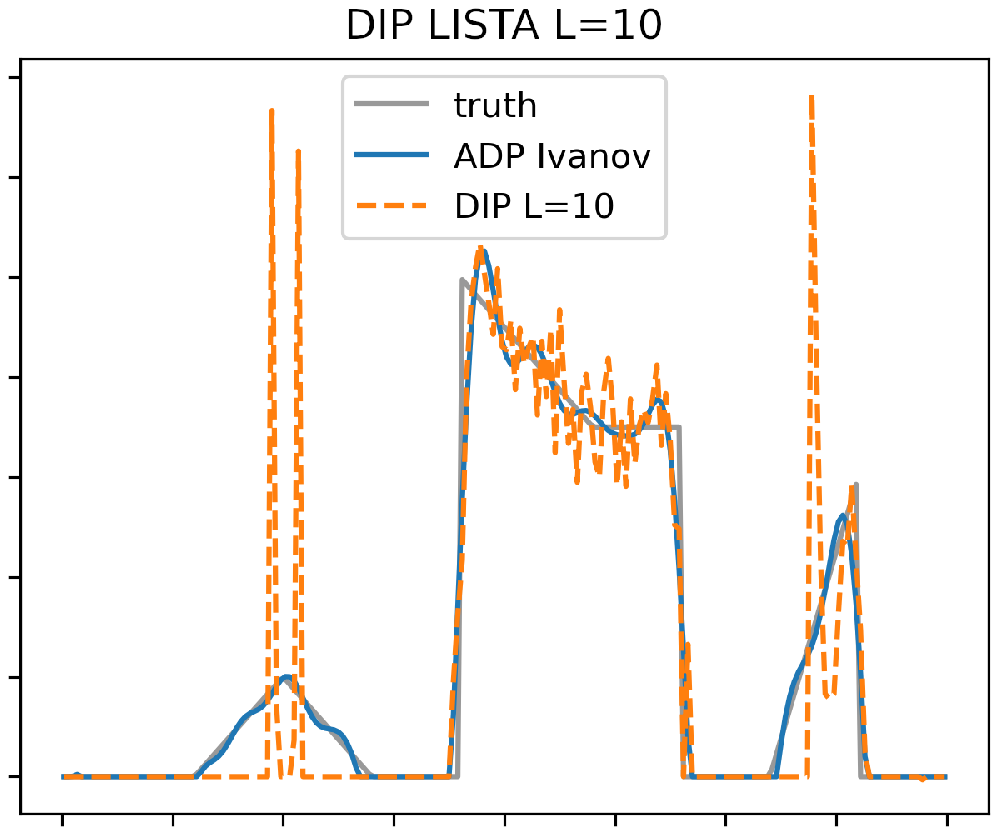}
\caption{Computation of ADP and DIP reconstructions via the IFT, via a DIP with LISTA network ($L=\infty$ and $L=10$) and via the equivalent Ivanov problem. The forward operator is $A_2$ (convolution). The given data has PSNR=45 due to additiv Gaussian noise. The regularization parameters $\alpha_1$, $\alpha_2$ are chosen for each example (row) separately but are the same for each method (column).}
\label{fig_numericways_conv}
\end{figure}

In most of the cases, the reconstructions of these both methods are looking quite similar to the actual ADP solution. But sometimes they contain artifacts (e.g.\ the peaks in figure \ref{fig_numericways_conv}, third row). It seems that there are some spots which are hard to reconstruct for the DIP methods and others are rather simple.
Besides, the ADP problem \eqref{eq_ADP} is not a convex minimization problem w.r.t.\ $B$. So there is no guarantee for the methods which do gradient descent (DIP LISTA $L=\infty$ and the IFT method) to converge towards the global minimizer. Figure \ref{fig_initialvalue} shows that the reconstructions of these methods are indeed dependent on the initial value $B_0$ of the algorithms. In contrast to that, the Ivanov problem from Theorem \ref{theo_equivalence} is convex (with the elastic net penalty term $R$). That's probably why the actual ADP solutions are the only ones which never contain strange artifacts and the only ones that are always quite good reconstructions of the ground truth.

\begin{figure}[ht]
    \centering
    \includegraphics[width=0.32\textwidth]{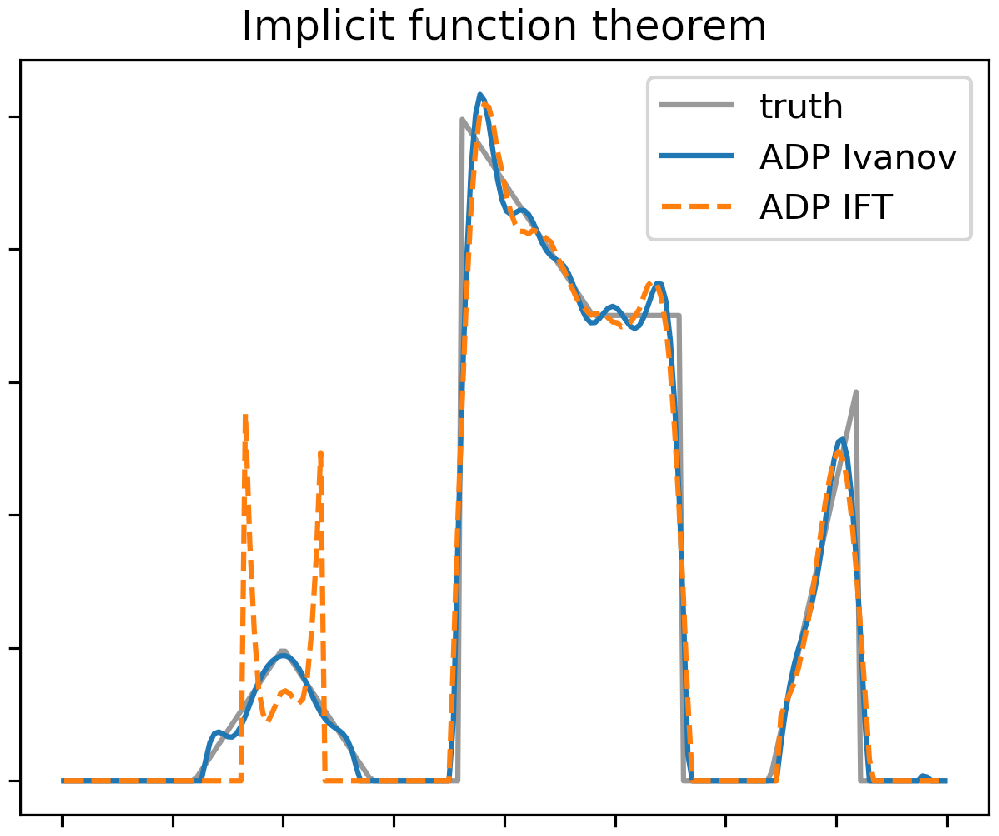}
    \includegraphics[width=0.32\textwidth]{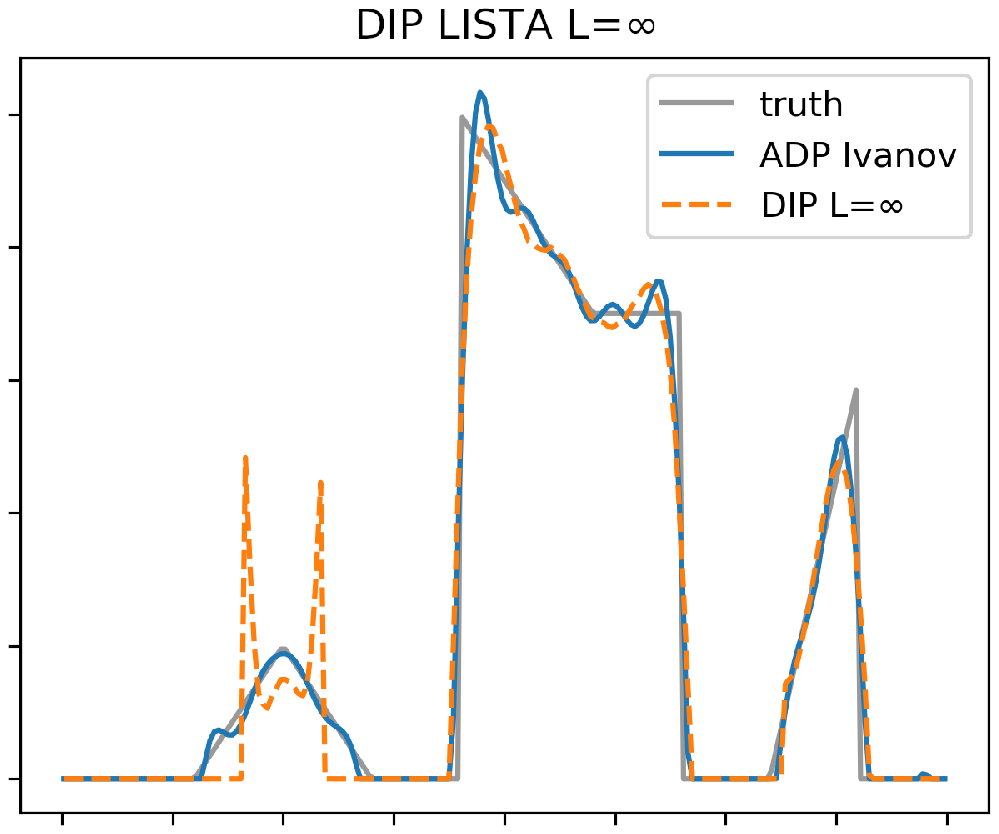}
    \caption{The same setting and methods as in figure \ref{fig_numericways_conv} but with different initial values $B_0$.}
    \label{fig_initialvalue}
\end{figure}

The easiest possibility to slightly improve the reconstrution quality is to apply early stopping. In doing so, the most severe artifacts in the reconstructions can be diminished. This case corresponds to the ADP-$\beta$ approach (see section~\ref{ch_earlystopping}), whose additional convex term $\beta \|B-A\|^2$ is a numerical advantage because it stabilizes the gradient descent for finding the minimizer. Indeed, adding the gradient of the $\beta$-term to the update in Algorithm \ref{algo_ift} (ADP IFT) can also diminish the severe artifacts. But we do not include experimental results about this, since the most interesting part is the comparison with the equivalent Ivanov problem, which doesn't exist for ADP-$\beta$.

The main conclusion is the numerical verification of the derivation of the ADP problem from the DIP approach. It is possible to use the theoretical analysis of the ADP problem for interpretations of the DIP approach because of the similarity between the reconstructions from the different numerical methods. However, the examples from figure~\ref{fig_initialvalue} illustrate that DIP can be formulated as a minimization problem \eqref{eq_DIP} but a numerical computed solution is not automatically a global minimizer of this problem. If early stopping is used, it is probably not even a local minimizer. Hence, there is a significant difference between the theoretical definition and the practical implementation of DIP.

%
%

\section{Conclusion}

ADP and ADP-$\beta$ were introduced as methods for solving ill-posed inverse problems in a typical Hilbert space setting (Assumption \ref{def_assumptions}). Both of them are motivated by considering DIP with a LISTA-like architecture. The main result is an equivalence of ADP to the classical method of Ivanov regularization.

We have proven existence, stability and convergence results for both ADP and ADP-$\beta$. The obtained regularization properties are comparable to the ones of classical methods like Tikhonov's. In principal, these results can be transferred to DIP with LISTA-like networks. But due to non-convexity of the DIP minimization problem, numerically computed DIP solutions can differ significantly from exact ADP solutions, although they are similar in many cases. We conclude that theoretical analyses of the DIP approach should consider the whole optimization process and not only the properties of the minimizer.

One very important part is the early stopping of the DIP optimization process. In the ADP setting, we incorporated this strategy with an additional penalty term, which resulted in the ADP-$\beta$ model. The effect of this regularization can be seen by comparing the convergence theorems of ADP and ADP-$\beta$. Theorem \ref{theo_convergencebeta} provides a parameter choice rule ($\alpha \sim \delta$) for ADP-$\beta$, which is a big advantage over ADP.

A generalization of the ADP regularization results to DIP with general convolutional neural networks (CNNs) would be very desirable. The LISTA architecture was suitable because of its similarity to proximal splitting algorithms and the possibility to interpret the output as a solution of a variational problem. Finding similar connections for general CNNs is harder. However in \cite{Celledoni_2021}, CNNs are used to model proximal mappings and in \cite{papyan2018}, CNNs are interpreted as algorithms for sparse coding. Besides, \cite{combettes2020} asserts that most common activation functions are in fact proximal mappings and they establish a theory for characterizing the fixed point sets of neural networks as solutions of variational inequalities. These directions could provide ideas for possible future extensions.

%
%

\section*{Acknowledgements}
I want to thank Dr.\ Daniel Otero Baguer, Prof.\ Peter Maaß, Dr.\ Tobias Kluth and many more colleagues from the University of Bremen and Dr.\ Yury Korolev from the University of Cambridge for helpful advice and feedback.

%
%

\appendix

\section{Proofs of theoretical results}


\subsection{Theorem \ref{theo_xBstetig}} \label{app_proof_xBstetig}

Continuity of $B \mapsto x(B)$.

\begin{proof}
Let $(B_k) \subset L(X,Y)$ be a sequence of operators with $B_k \to B \in L(X,Y)$. At first, we mention that
the sequence $(x(B_k))$ is bounded because
\begin{equation}
\begin{split}
    \alpha R(x(B_k)) \leqslant \frac{1}{2} \|B_k x(B_k) - y^\delta\|^2 + \alpha R(x(B_k)) 
    \leqslant \frac{1}{2} \|B_k x(B) - y^\delta\|^2 + \alpha R(x(B)) 
    \end{split}
\end{equation}
holds. 
Further, we can estimate
\begin{align}
\begin{split}
    & \quad \, \, \frac{1}{2} \|B x(B_k) - y^\delta\|^2 + \alpha R(x(B_k))\\
    &= \frac{1}{2} \|B_k x(B_k) - y^\delta + (B-B_k) \, x(B_k)\|^2 + \alpha R(x(B_k)) \\
    &\leqslant \frac{1}{2} (\|B_k x(B_k)  - y^\delta\| + \|(B-B_k) \,x(B_k)\|)^2 + \alpha R(x(B_k)) \\
    &=  \frac{1}{2} \|B_k x(B_k)  - y^\delta\|^2 + \alpha R(x(B_k))\\
    &\, +  \|(B-B_k) \, x(B_k)\| \mydot \left(\|B_k x(B_k)  - y^\delta\| + \frac{1}{2}\|(B-B_k) \, x(B_k)\|\right).
    \end{split}
\end{align}
Because of the boundedness of $(x(B_k))$ and the convergence $B_k \to B$, the term
\begin{equation}
    \|(B-B_k) \, x(B_k)\| \mydot \left(\|B_k x(B_k)  - y^\delta\| + \frac{1}{2}\|(B-B_k) \, x(B_k)\|\right)
\end{equation}
converges to zero. For the remaining terms, we can estimate
\begin{equation}
    \frac{1}{2} \|B_k x(B_k)  - y^\delta\|^2 + \alpha R(x(B_k))
    \leqslant \frac{1}{2} \|B_k x(B)  - y^\delta\|^2 + \alpha R(x(B)),
\end{equation}
and $\|B_k x(B)  - y^\delta\|^2$ converges to $\|B x(B)  - y^\delta\|^2$.
So $(x(B_k))$ is a minimizing sequence of the strongly convex functional $\frac{1}{2} \|B x  - y^\delta\|^2 + \alpha R(x)$ because $x(B)$ is the minimizer. By \cite[Theorem 1]{LOONEY1977835}, the minimizing sequence converges to the minimizer $x(B)$.
\end{proof}


\subsection{Lemma \ref{lem_existsB_allg}} \label{app_proof_existsB_allg}

First part of the equivalence theorem for ADP to Ivanov problems.

\begin{proof}
Let $\hat{x}, v, y^\delta, \alpha$ and $R$ be given according to the assumptions.
We have to find a linear operator $B$ such that 
\begin{equation}
    -B^*(B \hat{x} - y^\delta) \in \alpha \partial R(\hat{x}).
\end{equation}
holds.
Because of $v \in \partial R(\hat{x})$, we just try to solve the equation
\begin{equation} \label{eq_B_und_v}
- B^*(B \hat{x} - y^\delta) = \alpha v 
\end{equation}
for $B$. If $\hat{x} = 0$, solving would be trivial. Otherwise, we can decompose $v$ into
\begin{equation}
v = \mu \hat{x} + v_\bot \qquad \text{s.t. } \langle v_\bot, \hat{x} \rangle = 0.
\end{equation} 
Accordingly it is $\mu = {\langle v, \hat{x} \rangle}/{\|\hat{x}\|^2}$. With that, we can write the equation from above as
\begin{equation} \label{eq_B_und_v2}
B^*B \hat{x} + \alpha \mu \hat{x} + \alpha v_\bot = B^* y^\delta.
\end{equation}

We consider a linear operator $B \colon X \to Y$ of the form
\begin{equation}
B x = ( \sigma_1 \langle x, \hat{x} \rangle + \sigma_2 \langle x, v_\bot \rangle ) \mydot y^\delta \label{eq_Bsigma12}
\end{equation}
with two coefficients $\sigma_1$ and $\sigma_2$ to be determined later. Then, the adjoint operator is given by
\begin{equation}
B^* y = \langle y, y^\delta \rangle ( \sigma_1 \hat{x} + \sigma_2 v_\bot)
\end{equation}
and it holds
\begin{align}
B^*B \hat{x} &= B^*(( \sigma_1 \|\hat{x} \|^2 ) \mydot y^\delta) = \sigma_1^2 \|\hat{x}\|^2 \|y^\delta \|^2 \hat{x} + \sigma_1 \sigma_2 \|\hat{x}\|^2 \|y^\delta \|^2 v_\bot, \\
B^*y^\delta &= \sigma_1 \|y^\delta\|^2  \hat{x} + \sigma_2 \|y^\delta\|^2 v_\bot.
\end{align}
To fulfill \eqref{eq_B_und_v2}, we have to solve
\begin{equation}
\sigma_1^2 \|\hat{x}\|^2 \|y^\delta \|^2 \hat{x} + \sigma_1 \sigma_2 \|\hat{x}\|^2 \|y^\delta \|^2 v_\bot + \alpha \mu \hat{x} + \alpha v_\bot = \sigma_1 \|y^\delta\|^2  \hat{x} + \sigma_2 \|y^\delta\|^2 v_\bot.
\end{equation}
Because $\hat{x}$ and $v_\bot$ are orthogonal to each other, we get the two equations
\begin{align}
\sigma_1^2 \|\hat{x}\|^2 \|y^\delta \|^2  + \alpha \mu = \sigma_1 \|y^\delta\|^2, \label{eq_sigma1} \\
\sigma_1 \sigma_2 \|\hat{x}\|^2 \|y^\delta \|^2 + \alpha =  \sigma_2 \|y^\delta\|^2.\label{eq_sigma2}
\end{align}
Notice that \eqref{eq_sigma2} and the coefficient $\sigma_2$ could be ignored if $v_\bot = 0$ held.

Equation \eqref{eq_sigma1} can be solved for $\sigma_1$ with a quadratic formula, which leads to
\begin{equation}
\sigma_1 = \frac{1}{2\|\hat{x}\|^2} \pm \sqrt{\frac{1}{4\|\hat{x}\|^4} - \frac{\alpha \mu}{\|\hat{x}\|^2 \| y^\delta \|^2}}.
\end{equation}
Accordingly,
\begin{equation}
\frac{\alpha \mu}{\| y^\delta \|^2} \leqslant \frac{1}{4\|\hat{x}\|^2}
\end{equation}
must hold to get real solutions. We know from above that $\mu = {\langle v, \hat{x} \rangle}/{\|\hat{x}\|^2}$. If we insert this, we will see that this matches exactly the assumptions of the lemma.

Now, equation \eqref{eq_sigma2} has to be solved for $\sigma_2$. Excluding $\sigma_2$ leads to
\begin{equation}
\sigma_2 (\sigma_1 \|\hat{x}\|^2 \|y^\delta \|^2 - \|y^\delta\|^2) + \alpha = 0.
\end{equation}
If the term inside of the parenthesis doesn't equal zero, there will exist a solution $\sigma_2$. If the term equaled zero, equation \eqref{eq_sigma1} would lead to $\mu = 0$. But in this case, we could choose $\sigma_1 = 0$ (the quadratic formula allows two solutions), and then it is no problem to find a solution for $\sigma_2$, too.

By finding solutions for $\sigma_1$ and $\sigma_2$, we showed that the operator $B$ defined in \eqref{eq_Bsigma12} solves equation \eqref{eq_B_und_v}. So the lemma is proved.
\end{proof}


\subsection{Lemma \ref{lem_notexistsB_allg}} \label{app_proof_notexistsB_allg}

Second part of equivalence theorem for ADP to Ivanov problems.

\begin{proof}
Let $\hat{x}, y^\delta, \alpha$ and $R$ be given according to the assumptions. Assume there exists a linear operator $B$ such that 
\begin{equation}
0 \in B^*(B \hat{x} - y^\delta) + \alpha \partial R(\hat{x})
\end{equation}
holds. It follows
\begin{equation}
    v := - \frac{1}{\alpha} B^*(B \hat{x} - y^\delta) \in \partial R(\hat{x}).
\end{equation} 
We can calculate $\alpha \langle v,\hat{x}\rangle =  -\| B \hat{x}\|^2 + \langle y^\delta, B \hat{x} \rangle$.
So according to the assumptions,
\begin{equation}
 -\| B \hat{x}\|^2 + \langle y^\delta, B \hat{x} \rangle  > \frac{\|y^\delta\|^2}{4}
\end{equation}
must hold.
But with Young's inequality, we get
\begin{equation}
-\| B \hat{x}\|^2 + \langle y^\delta, B \hat{x} \rangle \leqslant -\| B \hat{x}\|^2 + \frac{1}{4} \|y^\delta\|^2 + \|B \hat{x}\|^2 = \frac{\|y^\delta\|^2}{4}.
\end{equation}
Obviously, this is a contradiction. That's why such an operator $B$ can't exist.
\end{proof}


\subsection{Lemma \ref{lem_parameters}} \label{app_proof_parameters}

Relation between the ADP parameter and the Tikhonov parameter of the equivalent problem.

\begin{proof}
Let $\hat{x}$ be the solution of the ADP problem \eqref{eq_adp_alpha}. Because of the equivalence to the Tikhonov method, $\hat{x}$ is the solution of \eqref{eq_tik_alpha} in the same time. Besides, $x(A)$ is the Tikhonov solution w.r.t.\ the parameter $\alpha_{\text{ADP}}$ of the inverse problem. Because of the minimizing properties of $\hat{x}$ and $x(A)$,
\begin{align}
    \frac{1}{2} \| A \hat{x} - y^\delta \|^2 &\leqslant \frac{1}{2} \| A x(A) - y^\delta \|^2, \\
    \frac{1}{2} \| A x(A) - y^\delta \|^2 + \frac{\alpha_{\text{ADP}}}{2} \|x(A)\|^2 &\leqslant \frac{1}{2} \| A \hat{x} - y^\delta \|^2 + \frac{\alpha_{\text{ADP}}}{2} \|\hat{x}\|^2
\end{align}
holds. It follows $\|x(A)\|^2 \leqslant \|\hat{x}\|^2$. Both $x(A)$ and $\hat{x}$ are Tikhonov solutions of the same problem (only with different parameters). So $\tilde{\alpha} \leqslant \alpha_{\text{ADP}}$ must hold because the norm of $\hat{x}$ is greater (or equal) than the norm of $x(A)$.

Now, we assume $\tilde{\alpha} = \alpha_{\text{ADP}}>0$. By Remark \ref{rem_l2_equivalence}, the problems
\begin{align}
    &\min_{x \in X} \frac{1}{2}\|Ax - y^\delta\|^2 + \frac{\alpha_{\text{ADP}}}{2} \|x\|^2,\\
    &\min_{x \in X} \frac{1}{2}\|Ax - y^\delta\|^2 \quad \text{s.t. } \|x\|^2 \leqslant \frac{\|y^\delta\|^2}{4\alpha_{\text{ADP}}} \label{eq_lemparam_ivanov}
\end{align}
are equivalent. The solution $\hat{x}$ fulfills
\begin{equation} \label{eq_lemparam_tikhonov}
    (A^* A + \alpha_{\text{ADP}} \mydot \mathrm{Id}) \hat{x} = A^* y^\delta
\end{equation}
and we assume $A^* y^\delta \neq 0$. Then, $\hat{x}$ must fulfill the side constraint of \eqref{eq_lemparam_ivanov} with equality, otherwise $A^*A \hat{x} = A^* y^\delta$ would hold, which is a contradiction. Accordingly we get $\alpha_{\text{ADP}}\|\hat{x}\|^2 = \|y^\delta\|^2/4$ and by computing the inner product of \eqref{eq_lemparam_tikhonov} with $\hat{x}$, it follows
\begin{equation} \label{eq_singularcondition}
     \|A \hat{x}\|^2 + \frac{\|y^\delta\|^2}{4} = \|A \hat{x}\|^2 + \alpha_{\text{ADP}} \|\hat{x}\|^2 = \langle A \hat{x}, y^\delta \rangle.
\end{equation}
If we then apply the Cauchy-Schwarz and Young's inequality, we get
\begin{equation}
    \langle A \hat{x}, y^\delta \rangle \leqslant \|A \hat{x}\| \mydot \|y^\delta\| \leqslant \|A \hat{x}\|^2 + \frac{\|y^\delta\|^2}{4},
\end{equation}
which means these inequalities must in fact hold as equalities. Therefore, $A \hat{x}$ and $y^\delta$ must be linear dependent (Cauchy-Schwarz) and $2\|A \hat{x}\| = \|y^\delta\|$ must hold (Young).
It follows
\begin{equation}
    A \hat{x} = \frac{1}{2} y^\delta.
\end{equation}
We can plug this into \eqref{eq_lemparam_tikhonov} and get
\begin{equation}
   \alpha_{\text{ADP}} \hat{x} = \frac{1}{2} A^* y^\delta.
\end{equation}
Accordingly $A A^* y^\delta =  \alpha_{\text{ADP}}  y^\delta$ holds, so $y^\delta$ is a singular vector of $A$.
\end{proof}


\subsection{Theorem \ref{theo_stabilityADP}} \label{app_proof_stabilityADP}

Stability of the ADP approach.

\begin{proof}
We follow some of the ideas of the proofs of \cite[Theorem 2.1]{Engl_1989} and \cite[Theorem 2]{Seidman_1989}.

Let $x_k$ and $\hat{x}$ be unique solutions of \eqref{eq_altFormtik} for $y^\delta=y_k, \hat{y}$ with $y_k \to \hat{y}$. The sequence $(x_k)$ is bounded, so there exists a weakly convergent subsequence $(x_{k_l})$, $x_{k_l} \rightharpoonup x_\infty$. For arbitrary $\varepsilon>0$ and $x \in X$ with $\|x\|^2 \leqslant \|\hat{y}\|^2 \mydot (4 \alpha)^{-1} - \varepsilon$, it holds
\begin{equation}
    \| A x_\infty - \hat{y}\| \leqslant \liminf_{l \to \infty} \|A x_{k_l} - y_{k_l}\| \leqslant \lim_{l \to \infty} \|A x - y_{k_l}\| = \|Ax - \hat{y}\|
\end{equation}
because $x_{k_l}$ minimizes the ADP problem w.r.t.\ $y_{k_l}$ and $x$ fulfills the side constraint for $l$ big enough.
With $\varepsilon \to 0$ and because of the uniqueness of the solutions, we obtain $x_\infty = \hat{x}$.
Arguing with a subsequence of a subsequence leads to the weak convergence $x_k \rightharpoonup \hat{x}$ of the whole sequence.

According to the assumptions, it holds $\|x_k\|^2 = \|\hat{y_k}\|^2 \mydot (4 \alpha)^{-1}$. So $y_k \to \hat{y}$ implies ${\|x_k\| \to \|\hat{x}\|}$ and together with the weak convergence, we finally obtain $x_k \to \hat{x}$.
\end{proof}


\subsection{Theorem \ref{theo_stabilitybeta}} \label{app_proof_stabilitybeta}

Stability of the ADP-$\beta$ approach.

\begin{proof} First, we note that the sequence $(g_k)$ is bounded in $W^{1,2}(\Omega)$. Hence, there exists at least one weakly convergent subsequence. For any subsequence with $g_k \rightharpoonup \hat{g}$, it holds
\begin{align}
    T(f, x_{g_k}) - y_k \, &\to \, T(f, x_{\hat{g}}) - \hat{y}
\end{align}
because of the arguments from Remark \ref{rem_betaexistence}.
For arbitrary $g \in W^{1,2}(\Omega)$,
\begin{equation}
\begin{split} \label{eq_stab_minimizer}
    &\quad \, \frac{1}{2}\| T(f, x_{\hat{g}}) - \hat{y}\|_{L^2}^2 + \beta \|\hat{g}-f\|_{W^{1,2}}^2 \\
    &\leqslant \liminf_{k \to \infty} \frac{1}{2}\| T(f,x_{g_k}) - y_k\|_{L^2}^2 + \beta \|g_k-f\|_{W^{1,2}}^2 \\
    &\leqslant \lim_{k \to \infty} \frac{1}{2}\| T(f, x_{g}) - y_k\|_{L^2}^2 + \beta \|g-f\|_{W^{1,2}}^2 \\
    &=  \frac{1}{2}\| T(f, x_{g}) - \hat{y}\|_{L^2}^2 + \beta \|g-f\|_{W^{1,2}}^2
\end{split}
\end{equation}
holds because of the minimizing property of $g_k$ w.r.t.\ $y_k$. Hence, $\hat{g}$ is a minimizer of \eqref{eq_ADP-param} w.r.t\ $\hat{y}$. 
If we choose $g=\hat{g}$, the first and the last line in \eqref{eq_stab_minimizer} coincide, and we get
\begin{equation}
    \lim_{k \to \infty} \frac{1}{2}\| T(f,x_{g_k}) - y_k\|_{L^2}^2 + \beta \|g_k-f\|_{W^{1,2}}^2 = \frac{1}{2}\| T(f, x_{\hat{g}}) - \hat{y}\|_{L^2}^2 + \beta \|\hat{g}-f\|_{W^{1,2}}^2.
\end{equation}
If follows $\lim_{k \to \infty} \|g_k-f\|_{W^{1,2}}^2 = \|\hat{g}-f\|_{W^{1,2}}^2$. Hence, $(g_k)$ converges by norm to $\hat{g}$.
\end{proof}


\subsection{Theorem \ref{theo_convergencebeta}} \label{app_proof_convergencebeta}

Convergence of the ADP-$\beta$ approach.

\begin{proof} According to \eqref{eq_sourcecondition}, we can choose $d = R(\hat{x}_\alpha^\delta) - R(x^\dagger) - \langle A^*w, \hat{x}_\alpha^\delta - x^\dagger\rangle$ and there exists an operator ${\hat{B} \in L(X,Y)}$ that fulfills $\hat{x}_\alpha^\delta = x(\hat{B})$.

Because of the minimizing property of $\hat{x}_\alpha^\delta$,
\begin{equation}
\alpha R(\hat{x}_\alpha^\delta) \leqslant \frac{1}{2}\|\hat{B}\hat{x}_\alpha^\delta- y^\delta\|^2 + \alpha R(\hat{x}_\alpha^\delta) \leqslant \frac{1}{2} \|\hat{B}x^\dagger - y^\delta\|^2 + \alpha R(x^\dagger).
\end{equation}
holds. If follows
\begin{align}
\begin{split}
    d &= R(\hat{x}_\alpha^\delta) - R(x^\dagger) - \langle A^*w, \hat{x}_\alpha^\delta - x^\dagger \rangle \leqslant \frac{1}{2 \alpha} \|\hat{B}x^\dagger - y^\delta\|^2 - \langle w, A\hat{x}_\alpha^\delta - y^\dagger \rangle \\
&\leqslant \frac{1}{2\alpha}\left(\|\hat{B}x^\dagger - Ax^\dagger\| + \|y^\dagger - y^\delta\|\right)^2 + \|w\| \|A\hat{x}_\alpha^\delta - y^\dagger\| \\
&\leqslant \frac{1}{2\alpha}\left(\|x^\dagger\| \|\hat{B}- A\| + \delta\right)^2 + \|w\| \|A\hat{x}_\alpha^\delta - y^\dagger\|.
\end{split}
\end{align}

We will show $\|\hat{B}-A\| = O(\delta)$ and $\|A  \hat{x}_\alpha^\delta - y^\dagger\| = O(\delta)$ to deduce $d = O(\delta)$ for $\alpha$ chosen proportional to $\delta$.
Because of the minimizing property of $\hat{B}$, we get
\begin{equation}
    \beta \mydot \|\hat{B}-A\|^2 \leqslant \frac{1}{2} \|A  x(\hat{B}) - y^\delta\|^2 + \beta \mydot \|\hat{B}-A\|^2 \leqslant  \frac{1}{2} \|A  x(A) - y^\delta\|^2.
\end{equation}
From standard convergence results of the Tikhonov method \cite[Theorem~4.4]{Hofmann_2007} or \cite[Theorem~2]{Burger_2004}, we get $\|A  x(A) - y^\delta\| = O(\delta)$. So $\|\hat{B}-A\| = O(\delta)$ holds.

Besides,
\begin{align}
     \|A\hat{x}_\alpha^\delta - y^\dagger\| \leqslant  \|A x(\hat{B}) - y^\delta\| + \|y^\delta - y^\dagger\|
    \leqslant  \|A x(A) - y^\delta\| + \delta
\end{align}
holds and we can use $\|A  x(A) - y^\delta\| = O(\delta)$ again. So $d = O(\delta)$ follows.
\end{proof}

\bibliographystyle{abbrv}
\bibliography{literature}

\end{document}